\documentclass[a4paper]{amsart}

\title{The dynamical hierarchy for Roelcke precompact Polish groups}
\author{Tom\'as Ibarluc\'ia}
\address{Universit\'e de Lyon \\Universit\'e Claude Bernard Lyon 1 \\CNRS UMR 5208, Institut Camille Jordan \\43 blvd. du 11 novembre 1918\\69622 Villeurbanne Cedex \\France}
\email{ibarlucia@math.univ-lyon1.fr}
\thanks{Research partially supported by GruPoLoCo (ANR-11-JS01-0008) and ValCoMo (ANR-13-BS01-0006)}

\usepackage{amsmath, amssymb, amsthm, amscd, enumerate, dsfont}

\usepackage[T1]{fontenc}
\usepackage{kpfonts}
\usepackage{a4wide}

\usepackage{stmaryrd}

\def\ZZ{\mathbb{Z}}
\def\QQ{\mathbb{Q}}
\def\RR{\mathbb{R}}
\def\AA{\mathcal{A}}
\def\BB{\mathcal{B}}
\def\CC{\mathcal{C}}

\DeclareMathOperator{\acl}{acl}
\DeclareMathOperator{\tp}{tp}
\DeclareMathOperator{\Aut}{Aut}
\DeclareMathOperator{\Iso}{Iso}
\DeclareMathOperator{\Homeo}{Homeo}
\DeclareMathOperator{\Hilb}{Hilb}
\DeclareMathOperator{\WAP}{WAP}
\DeclareMathOperator{\AP}{AP}
\DeclareMathOperator{\Asp}{Asp}
\DeclareMathOperator{\SUC}{SUC}
\DeclareMathOperator{\RUC}{RUC}
\DeclareMathOperator{\LUC}{LUC}
\DeclareMathOperator{\UC}{UC}
\DeclareMathOperator{\DEF}{Def}
\DeclareMathOperator{\Tame}{Tame}
\DeclareMathOperator{\Null}{Null}
\DeclareMathOperator{\NIP}{NIP}

\theoremstyle{plain}        \newtheorem{fact}{Fact}[section]
\theoremstyle{plain}        \newtheorem{theorem}[fact]{Theorem}
\theoremstyle{plain}        \newtheorem{lem}[fact]{Lemma}
\theoremstyle{plain}        \newtheorem{prop}[fact]{Proposition}
\theoremstyle{plain}        \newtheorem{cor}[fact]{Corollary}
\theoremstyle{definition}   \newtheorem{rem}[fact]{Remark} 
\theoremstyle{definition}   \newtheorem{defin}[fact]{Definition}
\theoremstyle{definition}   \newtheorem{note}[fact]{Note}
\theoremstyle{definition}   \newtheorem{question}[fact]{Question}
\theoremstyle{definition}   \newtheorem{example}[fact]{Example}

\makeatletter
\g@addto@macro\bfseries{\boldmath}
\makeatother

\usepackage[hidelinks]{hyperref}

\begin{document}
\begin{abstract} We study several distinguished function algebras on a Polish group $G$, under the assumption that $G$ is Roelcke precompact. We do this by means of the model-theoretic translation initiated by Ben Yaacov and Tsankov: we investigate the dynamics of $\aleph_0$-categorical metric structures under the action of their automorphism group. We show that, in this context, every strongly uniformly continuous function (in particular, every Asplund function) is weakly almost periodic. We also point out the correspondence between tame functions and $\NIP$ formulas, deducing that the isometry group of the Urysohn sphere is $\Tame\cap\UC$-trivial.\end{abstract}
\maketitle
\tableofcontents

\section*{Introduction}

In a series of recent papers, Glasner and Megrelishvili \cite{glameg06,glameg08,meg08,glameg12,glameg10} have studied different classes of functions on topological dynamical systems, arising from compactifications with particular properties. Thus, for example, a real-valued continuous function on a $G$-space $X$ might be \emph{almost periodic}, \emph{Hilbert-representable}, \emph{weakly almost periodic}, \emph{Asplund-representable} or \emph{tame}, and this classes form a hierarchy $$\AP(X)\subset\Hilb(X)\subset\WAP(X)\subset\Asp(X)\subset\Tame(X)\subset\RUC(X)$$
of subalgebras of the class of \emph{right uniformly continuous} functions.
These algebras can be defined in different ways. The latter coincides with the class of functions that can be in some sense represented through a Banach space, and from this point of view the previous subalgebras can be identified, respectively, with the cases when the Banach space is asked to be Euclidean, Hilbert, reflexive, Asplund or Rosenthal.
When $X=G$ and the action is given by group multiplication, functions might also be \emph{left uniformly continuous}, and if they are simultaneously in $\RUC(G)$ they form part of the algebra $\UC(G)$ of \emph{Roelcke uniformly continuous} functions.

We study these algebras for the case of Roelcke precompact Polish groups, by means of the model-theoretic translation developed by Ben Yaacov and Tsankov \cite{bentsa}. As established in their work, Roelcke precompact Polish groups are exactly those arising as automorphism groups of $\aleph_0$-categorical metric structures. Moreover, one might turn continuous functions on the group into definable predicates on the structure. Under this correlation, the authors showed, weakly almost periodic functions translate into \emph{stable} formulas: a most studied concept of topological dynamics leads to one of the crucial notions of model theory. This provides a unified understanding of several previously studied examples: the permutation group $S(\mathbb{N})$, the unitary group $\mathcal{U}(\ell^2)$, the group of measure preserving transformations of the unit interval $\Aut(\mu)$, the group $\Aut(RG)$ of automorphisms of the random graph or the isometry group $\Iso(\mathbb{U}_1)$ of the Urysohn sphere, among many other ``big'' groups, are automorphism groups of $\aleph_0$-categorical structures, thus Roelcke precompact. In the first three cases the structures are stable, thus $\WAP(G)=\UC(G)$: their $\WAP$ and Roelcke compactifications coincide. Using model-theoretic insight, the authors were able to prove for example that, whenever the latter is the case, the group $G$ is totally minimal.

The so-called \emph{dynamical hierarchy} presented above has been partially described for some of the habitual examples. For the groups $S(\mathbb{N})$, $\mathcal{U}(\ell^2)$ or $\Aut(\mu)$ we have in fact $\Hilb(G)=\UC(G)$; see \cite[\textsection 6.3--6.4]{glameg13}. From \cite[\textsection 6]{bentsa} we know, for instance, that the inclusion $\WAP(G)\subset\UC(G)$ is strict for the group $\Aut(\QQ,<)$ of monotone bijections of the rationals. More drastically, Megrelishvili \cite{meg01} had shown that the group $H_+[0,1]$ of orientation preserving homeomorphisms of the unit interval, also Roelcke precompact, has a trivial $\WAP$-compactification: $\WAP(G)$ is the algebra of constants; in \cite[\textsection 10]{glameg08} this conclusion was extended to the algebra $\Asp(G)$ (and indeed to the algebra $\SUC(G)$ of \emph{strongly uniformly continuous functions}, containing $\Asp(G)$). The same was established for the group $\Iso(\mathbb{U}_1)$. If one drops the requirement of Roelcke precompactness, all inclusions in the hierarchy are known to be strict in appropriate examples.

We show that in fact $\WAP(G)=\Asp(G)=\SUC(G)$ for every Roelcke precompact Polish group~$G$. In addition, we observe that Roelcke uniformly continuous tame functions correspond to $\NIP$ formulas on the model-theoretic side. Thus, for instance, $\Asp(G)\subsetneq\Tame(G)\cap\UC(G)=\UC(G)$ for $G=\Aut(\QQ,<)$, while $\WAP(G)=\Tame(G)\cap\UC(G)\subsetneq\UC(G)$ for $G=\Aut(RG)$ or $G=\Homeo(2^\omega)$. We also deduce that the $\Tame\cap\UC$-compactification of $\Iso(\mathbb{U}_1)$ is trivial.

Our approach is model-theoretic, and we shall assume some familiarity with continuous logic as presented in \cite{benusv10} or \cite{bbhu08}; nevertheless, we give an adapted introduction to $\aleph_0$-categorical metric structures that we hope can be helpful to an interested reader with no background in logic. We will mainly study the dynamics of $\aleph_0$-categorical structures, then derive the corresponding conclusions for their automorphism groups.

The algebra $\Hilb(G)$ will not be addressed in this paper. Unlike the properties of stability and dependence, which can be studied locally (that is, formula-by-formula), the model-theoretic interpretation of the algebra $\Hilb(G)$ presents a different phenomenon, and will be considered in a future work.

\medskip

\noindent\textbf{Acknowledgements.} I am very much indebted to Ita\"i Ben Yaacov, who introduced me to his work with Todor Tsankov and asked whether a topological analogue of model-theoretic dependence could be found. I am grateful to Michael Megrelishvili for valuable discussions and observations, particularly Theorem \ref{megrelishvili} below. I want to thank Eli Glasner and Adriane Ka\"ichouh for their interest in reading a preliminary copy of this article and for their comments. Finally, I thank the anonymous referee for his detailed suggestions and corrections; they helped to improve significantly the exposition of this paper.

\noindent\hrulefill

\section{The setting and basic facts}

\subsection{$G$-spaces and compactifications} Most of the material on topology in this and subsequent sections comes from the works of Glasner and Megrelishvili referred to in the introduction.

A \emph{$G$-space} $X$ is given by a continuous left action of a topological group $G$ on a topological space~$X$. Then $G$ acts as well on the space $\CC(X)$ of continuous bounded real-valued functions on~$X$, by $gf(x)=f(g^{-1}x)$. If $X$ is not compact, however, the action on $\CC(X)$ need not be continuous for the topology of the uniform norm on $\CC(X)$. The functions $f\in\CC(X)$ for which the orbit map $g\in G\mapsto gf\in Gf\subset\CC(X)$ is norm-continuous are called \emph{right uniformly continuous} ($\RUC$). That is, $f\in\RUC(X)$ if for every $\epsilon>0$ there is a neighborhood $U$ of the identity of $G$ such that $$|f(g^{-1}x)-f(x)|<\epsilon$$ for all $x\in X$ and $g\in U$. When $X=G$ is considered as a $G$-space with the regular left action, we also have the family $\LUC(G)$ of \emph{left uniformly continuous functions}, where the condition is that $|f(xg)-f(x)|$ be small for all $x\in G$ and $g$ close to the identity. The intersection $\UC(G)=\RUC(G)\cap\LUC(G)$ forms the algebra of \emph{Roelcke uniformly continuous functions} on $G$. The family $\RUC(X)$ is a uniformly closed $G$-invariant subalgebra of $\CC(X)$, and the same is true for $\LUC(G)$ and $\UC(G)$ in the case $X=G$.

If $X$ is compact, then $\RUC(X)=\CC(X)$; in the case $X=G$, $\UC(G)=\CC(G)$. Moreover, recall that a compact Hausdorff space $X$ admits a unique compatible uniformity (see, for example, \cite[II, \textsection 4, \textnumero 1]{bourbakiTG}), and that any continuous function from $X$ to another uniform space is automatically uniformly continuous.

\begin{note}\label{metric setting} Our spaces, when not compact, will be metric, and $G$ will act on $X$ by uniformly continuous transformations (in practice, by isometries). In this case, we will usually restrict our attention to those functions $f\in\RUC(X)$ that are also uniformly continuous with respect to the metric on $X$; we denote this family of functions by $\RUC_u(X)$. It is a uniformly closed $G$-invariant subalgebra. The same subscript $u$ might be added to the other function algebras in the dynamical hierarchy, in order to keep this restriction in mind.

Our groups will be Polish. When we take $X=G$, we assume that a left-invariant, compatible, bounded metric $d_L$ on $G$ has been fixed; its existence is ensured by Birkhoff--Kakutani theorem, see for example \cite[p.~28]{berberian}. The subscript $u$ will then refer to this metric, and one should notice that $\RUC_u(G)=\UC(G)$. The algebra $\SUC(G)$, containing $\Asp(G)$ (both to be defined later), is always a subalgebra of $\UC(G)$ (see Section \ref{wap=suc}); in particular, $\SUC_u(G)=\SUC(G)$ and $\Asp_u(G)=\Asp(G)$. As pointed out to us by M. Megrelishvili, this is not the case for the algebra $\Tame(G)$ (see the discussion after Theorem~\ref{megrelishvili}), so we will mind the distinction between $\Tame(G)$ and $\Tame_u(G)=\Tame(G)\cap\UC(G)$.

From the equality $\RUC_u(G)=\UC(G)$ we see that $\RUC_u(G)$ does not depend on the particular choice of $d_L$. Thus, so far, we could omit the metric $d_L$ and consider simply the natural uniformities on $G$ (see for instance \cite[III, \textsection 3, \textnumero 1]{bourbakiTG} for an explanation of these). However, our approach will require to consider metric spaces, and in fact complete ones. This is why we will consider the space $(G,d_L)$, and mainly its completion $\widehat{G}_L=\widehat{(G,d_L)}$, which is naturally a $G$-space. We remark that the restriction map $\RUC_u(\widehat{G}_L)\to\UC(G)$ is a norm-preserving $G$-isomorphism.
\end{note}

A \emph{compactification} of a $G$-space $X$ is a continuous $G$-map $\nu\colon X\to Y$ into a compact Hausdorff $G$-space $Y$, whose range is dense in $Y$. In our context it will be important to consider compactifications that are uniformly continuous: in this case we shall say, to make the distinction, that $\nu$ is a \emph{$u$-compactification} of~$X$. A function $f\in\CC(X)$ \emph{comes from a compactification} $\nu\colon X\to Y$ if there is $\tilde{f}\in\CC(Y)$ such that $f=\tilde{f}\nu$; note that the extension $\tilde{f}$ is unique.

If $f$ comes from a compactification of $X$, then certainly $f\in\RUC(X)$. The converse is true. In fact, there is a canonical one-to-one correspondence between compactifications of $X$ and uniformly closed $G$-invariant subalgebras of $\RUC(X)$ (a subalgebra is always assumed to contain the constants). The subalgebra $\AA_\nu$ corresponding to a compactification $\nu\colon X\to Y$ is given by the family of all functions $f\in\CC(X)$ that come from~$\nu$. Conversely, the compactification $X^\AA$ corresponding to one such algebra $\AA\subset\RUC(X)$ is the space $X^\AA$ of characters of $\AA$ together with the map $\nu_\AA\colon X\to X^\AA$, $\nu_\AA(x)=\{f\in\AA:f(x)=0\}$. We recall that the topology on $X^\AA$ is generated by the basic open sets $U_{f,\delta}=\{p\in X^\AA:|\tilde{f}(p)|<\delta\}$ for $f\in\AA$ and $\delta>0$; here, $\tilde{f}(p)$ is the unique constant $r\in\RR$ such that $f-r\in p$.

In this way we always have the equality $\AA=\AA_{\nu_\AA}$ and a unique $G$-homeomorphism $j_\nu\colon X^{\AA_\nu}\to Y$ with $\nu=j_\nu\nu_{\AA_\nu}$. In particular, if $f\in\AA$, then $f$ comes from $\nu_\AA$ (and the extension $\tilde{f}\in\CC(X^\AA)$ is defined as above). Finally, the correspondence is functorial, in the sense that inclusions $\AA\subset\BB$ of subalgebras correspond bijectively to continuous $G$-maps $j\colon X^\BB\to X^\AA$ such that $\nu_\AA=j\nu_\BB$. When we say that a given compactification is minimal or maximal within a certain family, we refer to the order induced by these morphisms; in the previous situation, for example, $\nu_\BB$ is larger than~$\nu_\AA$.

More details on this correspondence can be found in \cite[IV, \textsection 5]{vriesElements}, particularly Theorem 5.18 (though the construction given there is quite different, not based on maximal ideal spaces; for the basics on maximal ideal spaces see \cite[VII, \textsection 8]{conwayFunctional}).

We point out here that the correspondence restricts well to our metric setting, namely, it induces a one-to-one correspondence between $u$-compactifications of $X$ and uniformly closed $G$-invariant subalgebras of $\RUC_u(X)$. Of course, if $\nu$ is a $u$-compactification of $X$ then any function coming from $\nu$ is uniformly continuous, so $A_\nu\subset\RUC_u(X)$. Conversely, we have the following.

\begin{fact} If $\AA$ is a uniformly closed $G$-invariant subalgebra of $\RUC_u(X)$ then $\nu_\AA\colon X\to X^\AA$ is uniformly continuous.\end{fact}
\begin{proof}
Suppose to the contrary that there is an entourage $\epsilon$ of the uniformity of $X^\AA$ such that for every $n$ there are $x_n,y_n\in X$ with distance $d(x_n,y_n)<1/n$ but such that $(\nu_\AA(x_n),\nu_\AA(y_n))\notin\epsilon$. We can assume the entourage is of the form $\epsilon=\bigcup_{i<k}U_i\times U_i$ for some cover of $X^\AA$ by basic open sets $$U_i=\{p\in X^\AA:|\tilde{f}_i(p)|<\delta_i\}$$ given by functions $f_i\in\AA$ and positive reals $\delta_i$.

Passing to a subnet we can assume that $\nu_\AA(x_n)$ converges to $p\in X^\AA$, say $p\in U_i$ for some $i<k$. Since $\AA$ is contained in $\RUC_u(X)$ (not merely in $\RUC(X)$) for $n$ big enough we have $|f_i(x_n)-f_i(y_n)|<\frac{1}{2}(\delta_i-|\tilde{f}_i(p)|)$, and also $|f_i(x_n)-\tilde{f}_i(p)|<\frac{1}{2}(\delta_i-|\tilde{f}_i(p)|)$. Thus for the same $n$ we have $|f_i(x_n)|<\delta_i$ and $|f_i(y_n)|<\delta_i$. This implies $(\nu_\AA(x_n),\nu_\AA(y_n))\in\epsilon$, a contradiction.
\end{proof}

\begin{rem}\label{compactifications of G} Let $G$ be a Polish group. Every $u$-compactification of $G$ factorizes through the left completion $\widehat{G}_L$, and we have a canonical one-to-one correspondence between $u$-compactifications of $G$ and of $\widehat{G}_L$.\end{rem}

The maximal $u$-compactification of a Polish group $G$, that is, the compactification $G^{\UC}$ associated to the algebra $\UC(G)$, is called the \emph{Roelcke compactification} of $G$. If we fix any $g\in G$, the function $d_g(h)=d_L(g,h)$ is in $\UC(G)$. This implies that the compactification $G\to G^{\UC}$ is always a topological embedding.

On the other hand, for any $f\in\RUC(X)$ there is a minimal compactification of $X$ from which $f$ comes, namely the one corresponding to the closed unital algebra generated by the orbit $Gf$ in $\CC(X)$. It is called the \emph{cyclic $G$-space} of $f$, and denoted by $X_f$.

An important part of the project developed in \cite{glameg06,glameg12,glameg13} has been to classify the dynamical systems (and particularly their compactifications) by the possibility of representing them as an isometric action on a ``good'' Banach space. Although we will not make use of it in the present paper, the precise meaning of a \emph{representation} of a $G$-space $X$ on a Banach space $V$ is given by a pair $$h\colon G\to\Iso(V),\ \alpha\colon X\to V^*,$$ where $h$ is a continuous homomorphism and $\alpha$ is a weak$^*$-continuous bounded $G$-map with respect to the dual action $G\times V^*\to V^*$, $(g\phi)(v)=\phi(h(g)^{-1}(v))$. The topology on $\Iso(V)$ is that of pointwise convergence. The representation is \emph{faithful} if $\alpha$ is a topological embedding.

For a family $\mathcal{K}$ of Banach spaces, a $G$-space $X$ is \emph{$\mathcal{K}$-representable} if it admits a faithful representation on a member $V\in\mathcal{K}$, and it is \emph{$\mathcal{K}$-approximable} if it can be topologically $G$-embedded into a product of $\mathcal{K}$-representable $G$-spaces.

\subsection{Roelcke precompact Polish groups} Following Uspenskij \cite[\textsection 4]{uspComp}, the infimum of the left and right uniformities on a Polish group $G$ is called the \emph{Roelcke uniformity} of the group. Accordingly, $G$ is \emph{Roelcke precompact} if its completion with respect to this uniformity is compact ---and thus coincides with the Roelcke compactification of $G$ as defined above. This translates to the condition that for every non-empty neighborhood $U$ of the identity there is a finite set $F\subset G$ such that $UFU=G$.

Let $G$ be a Polish group acting by isometries on a complete metric space $X$. Given a point $x\in X$, we denote by $[x]=\overline{Gx}$ the closed orbit of $x$ under the action. Then, we define the metric quotient $X\sslash G$ as the space $\{[x]:x\in X\}$ of closed orbits endowed with the induced metric $d([x],[y])=\inf_{g\in G}d(gx,y)$.

In the rest of the paper, given a countable (possibly finite) set $\alpha$, we will identify it with an ordinal $\alpha\leq\omega$ and consider the power $X^\alpha$ as a metric $G$-space with the distance $d(x,y)=\sup_{i<\alpha}2^{-i}d(x_i,y_i)$ and the diagonal action $gx=(gx_i)_{i<\alpha}$. Of course, the precise choice of the distance is arbitrary and we will only use that it is compatible with the product uniformity and that the diagonal action is by isometries.

The action of $G$ on $X$ is \emph{approximately oligomorphic} if the quotients $X^\alpha\sslash G$ are compact for every $\alpha<\omega$ (equivalently, for $\alpha=\omega$). Then, Theorem 2.4 in \cite{bentsa} showed the following.

\begin{theorem}\label{Roelcke bentsa} A Polish group $G$ is Roelcke precompact if and only if the action of $G$ on its left completion $\widehat{G}_L$ is approximately oligomorphic or, equivalently, if $G$ can be embedded in the group of isometries of a complete metric space $X$ in such a way that the induced action of $G$ on $X$ is approximately oligomorphic.\end{theorem}

Recall that the group of isometries of a complete metric space is considered as a Polish group with the topology of pointwise convergence.

Roelcke precompact Polish groups provide a rich family of examples of topological groups with interesting dynamical properties. By means of the previous characterization, Ben Yaacov and Tsankov initiated the study of these groups from the viewpoint of continuous logic.

\subsection{$\aleph_0$-categorical metric structures as $G$-spaces} Thus we turn to logic. We present the basic concepts and facts of the model theory of metric structures. About the general theory we shall be terse, and we refer the reader to the thorough treatments of \cite{benusv10} and \cite{bbhu08}; in fact, we will mostly avoid the syntactical aspect of logic. Instead, we will give precise topological reformulations for the case of $\aleph_0$-categorical structures. At the same time, we explain the relation to the dynamical notions introduced before.

A metric first-order structure is a complete metric space $(M,d)$ of bounded diameter together with a family of distinguished \emph{basic predicates} $f_i\colon M^{n_i}\to\RR$ ($n_i<\omega$), $i\in I$, which are uniformly continuous and bounded. (The structure may also have distinguished elements and basic functions from finite powers of $M$ into $M$, as is the case of the boolean algebra $\BB$ considered in the examples; but these can be coded with appropriate basic predicates.) An automorphism of the structure is an isometry $g\in\Iso(M)$ such that each basic predicate $f_i$ is invariant for the diagonal action of $g$ on $M^{n_i}$, that is, $f_i(gx)=f_i(x)$ for all $x\in M^{n_i}$. For a separable structure $M$, the space $\Aut(M)$ of all automorphisms of $M$ is a Polish group under the topology of pointwise convergence.

If $M$ is separable and isomorphic to any other separable structure with the same \emph{first-order properties}, then $M$ is \emph{$\aleph_0$-categorical}. A classical result in model-theory (see \cite{bbhu08}, Theorem 12.10) implies that this is equivalent to say that $M$ is separable and the action of $\Aut(M)$ on $M$ is approximately oligomorphic. In particular, by Theorem \ref{Roelcke bentsa}, $\Aut(M)$ is Roelcke precompact.

The structure $M$ is \emph{classical} if $d$ is the Dirac distance and the basic predicates are $\{0,1\}$-valued. In this case, $M$ is $\aleph_0$-categorical if and only if it is countable and the action of $\Aut(M)$ on $M$ is \emph{oligomorphic}, i.e.\ the quotients $M^n\sslash\Aut(M)$ are finite for every $n<\omega$.

A \emph{definable predicate} is a function $f\colon M^\alpha\to\RR$, with $\alpha$ a countable set, constructed from the basic predicates and the distance by continuous combinations, rearranging of the variables, approximate quantification (i.e.\ suprema and infima) and uniform limits. Every definable predicate is $\Aut(M)$-invariant, uniformly continuous and bounded. If $M$ is $\aleph_0$-categorical, then $f\colon M^\alpha\to\RR$ is a definable predicate if and only if it is continuous and $\Aut(M)$-invariant; see for example \cite{benkai13}, Proposition 1. 

\begin{defin} In this paper, we shall use the term \emph{formula} to denote a definable predicate in two countable sets of variables, i.e.\ a function $f\colon M^\alpha\times M^\beta\to\RR$, for countable sets $\alpha,\beta$, which is a definable predicate once we rewrite the domain as a countable power of $M$. We will denote it by $f(x,y)$ to specify the two variables of the formula. Given a formula $f(x,y)$ and an a parameter $a\in M^\alpha$, we denote by $f_a\in\CC(M^\beta)$ the continuous function defined by $f_a(b)=f(a,b)$. When we make no reference to $\alpha$ or $\beta$, we will assume that $\alpha=\omega$ and $\beta=1$.\end{defin}

Whenever we talk of a metric structure $M$ as a $G$-space, we understand that the group is $G=\Aut(M)$ and that it acts on $M$ in the obvious way. This $G$-space comes with a distinguished function algebra: the family of functions of the form $f_a$ for a formula $f(x,y)$ and a parameter $a\in M^\omega$. We will denote it by $\DEF(M)$, and it is in fact a uniformly closed $G$-invariant subalgebra of $\CC(M)$. More generally, if $a\in A^\omega$ for a subset $A\subset M$ (and the variable $y$ is of any length $\beta$), we will say that $f_a$ is an \emph{$A$-definable predicate} in the variable $y$. A $\emptyset$-definable predicate is just a definable predicate. The family of $A$-definable predicates in $y$ is clearly a subalgebra of $\CC(M^\beta)$, which is uniformly closed as the following shows.

\begin{fact} A uniform limit of $A$-definable predicates is an $A$-definable predicate.\end{fact}
\begin{proof} Say we have formulas $f^n(x,y)$ and parameters $a_n\in A^\omega$ such that $f^n_{a_n}$ converges uniformly; without loss of generality we can assume that the tuples are the same, say $a=a_n$. Passing to a subsequence we can assume that $f^n_{a_n}$ converges fast enough, then define $f(x,y)$ as the \emph{forced limit} of the formulas $f^n(x,y)$ (see \cite[\textsection 3.2]{benusv10}, and compare with Lemma 3.11 therein). Then the limit of the predicates $f^n_a$ is~$f_a$.\end{proof}

The starting point for our analysis is the following observation, based on the ideas from \cite[\textsection 5]{bentsa}.

\begin{prop}\label{Form=RUC} For a metric structure $M$ we have $\DEF(M)\subset\RUC_u(M)$. If $M$ is $\aleph_0$-categorical, then moreover $\DEF(M)=\RUC_u(M)$.\end{prop}
\begin{proof} For the first part consider a formula $f(x,y)$ together with a parameter $a\in M^\omega$. Take a neighborhood $U$ of the identity such that $d(a,ga)<\Delta_f(\epsilon)$ for $g\in U$, where $\Delta_f$ is a modulus of uniform continuity for $f(x,y)$. Thus $\|gf_a-f_a\|=\|f_{ga}-f_a\|<\epsilon$ whenever $g\in U$. This shows that every $f_a\in\DEF(M)$ is in $\RUC_u(M)$.

Now let $h\in\RUC_u(M)$, and set $a\in M^\omega$ to enumerate a dense subset of $M$. We define $f\colon Ga\times M\to\RR$ by $$f(ga,b)=gh(b)=h(g^{-1}b).$$ This is well defined because $a$ is dense in $M$; note also that $f$ is $G$-invariant and uniformly continuous. Indeed, we have $$|f(ga,b)-f(g'a,b')|\leq |gh(b)-gh(b')|+|gh(b')-g'h(b')|.$$ The first term on the right side is small if $b$ and $b'$ are close: simply observe that $d(g^{-1}b,g^{-1}b')=d(b,b')$, so we use the uniform continuity of $h$. For the second, given $\epsilon>0$ there is a neighborhood $U$ of the identity of $G$ such that $\|gh-g'h\|<\epsilon$ whenever $g^{-1}g'\in U$, because $h$ is $\RUC$; since $a$ is dense, there is $\delta>0$ such that $d(ga,g'a)<\delta$ implies $g^{-1}g'\in U$; thus if $d(ga,g'a)<\delta$ we have $|gh(b')-g'h(b')|<\epsilon$.

This means that $f$ can be extended continuously to $[a]\times M$ (we recall the notation $[a]=\overline{Ga}$). The extension remains $G$-invariant, so we may regard $f$ as defined on $([a]\times M)\sslash G$, which is a closed subset of the metric space $(M^\omega\times M)\sslash G$. Then we can apply Tietze extension theorem to get a continuous extension to $(M^\omega\times M)\sslash G$. Composing with the projection we get a $G$-invariant continuous function $$f\colon M^\omega\times M\to\RR.$$ Finally, if $M$ is $\aleph_0$-categorical then the $G$-invariant continuous function $f$ is in fact a formula $f(x,y)$. Hence we have $h=f_a$, as desired.\end{proof}

In light of this result, if $M$ is $\aleph_0$-categorical, we can attempt to study the subalgebras of $\RUC_u(M)$ with model-theoretic tools; this is our aim.

Our conclusions will translate easily from structures to groups, the latter being the main subject of interest from the topological viewpoint. Indeed, if $G$ is a Polish group, there is a canonical construction (first described by J. Melleray in \cite[\textsection 3]{mel10}) that renders the left completion $M=\widehat{G}_L$ a metric first-order structure with automorphism group $\Aut(M)=G$. It suffices to take for $I$ the set of all closed orbits in all finite powers of $\widehat{G}_L$, that is $I=\bigsqcup_{n<\omega}M^n\sslash G$, then define the basic predicates $P_i\colon M^{n_i}\to\RR$ (if $i\in M^{n_i}\sslash G$) as the distance functions to the corresponding orbits: $P_i(y)=\inf_{x\in i}d(x,y)$. By Theorem \ref{Roelcke bentsa}, if $G$ is Roelcke precompact then $G$ acts approximately oligomorphically on its left completion and hence $M$ is an $\aleph_0$-categorical structure. In addition we have the natural norm-preserving $G$-isomorphism $\RUC_u(\widehat{G}_L)\simeq\UC(G)$. By this means, our conclusions about the dynamics of $\aleph_0$-categorical structures will carry immediately to Roelcke precompact Polish groups.

Nevertheless, for the analysis of the examples done in Section~\ref{Examples} we shall use the approach initiated in \cite[\textsection 5--6]{bentsa} for the study of $\WAP(G)$. That is, we will describe the functions on $G$ in terms of the formulas of the ``natural'' structure $M$ for which $G=\Aut(M)$ (see particularly Lemma 5.1 of the referred paper). To this end we have the following version of Proposition \ref{Form=RUC}.

\begin{prop}\label{Form=UC} Let $M$ be a metric structure, $G=\Aut(M)$. If $f(x,y)$ is an arbitrary formula and $a,b$ are tuples from $M$ of the appropriate length, then the function $g\mapsto f(a,gb)$ is in $\UC(G)$. If $M$ is $\aleph_0$-categorical and $h\in\UC(G)$, then there are a formula $f(x,y)$ in $\omega$-variables $x,y$ and a parameter $a\in M^\omega$ such that $h(g)=f(a,ga)$ for every $g\in G$.\end{prop}
\begin{proof} The proof of Proposition \ref{Form=RUC} can be adapted readily. Alternatively, we remark that if $a\in M^\omega$ enumerates a dense subset of $M$ then $[a]$ can be identified with $\widehat{G}_L$ (see \cite{bentsa}, Lemma 2.3). Thus the basic predicates $P_i\colon (\widehat{G}_L)^{n_i}\to\RR$ defined above are simply the restrictions to $[a]^{n_i}$ of the functions $f_i(y)=\inf_{x\in i}d(x,y)\colon (M^\omega)^{n_i}\to\RR$, which are definable predicates if $M$ is $\aleph_0$-categorical; similarly for the general definable predicates on $\widehat{G}_L$. The second claim in the statement then follows from this together with the identifications $\UC(G)\simeq\RUC_u(\widehat{G}_L)=\DEF(\widehat{G}_L)$.\end{proof}

\subsection{Types, extensions, indiscernibles} Before we go on, we recall some additional terminology from model theory that we use in our expositions and proofs. Most of it could be avoided if we decided to give a prevailingly topological presentation of our results, but we have chosen to emphasize the interplay between the two domains.

Let $M$ be a metric structure, $A\subset M$ a subset and let $y$ be a variable of length $\beta$. A \emph{(complete) type over $A$ (in $M$) in the variable $y$} can be defined as a maximal ideal of the uniformly closed algebra of $A$-definable predicates of $M$ in the variable $y$. The type over $A$ of an element $b\in M^\beta$ is defined by $$\tp(b/A)=\{f_a:a\in A^\omega,f(x,y)\text{ a formula with }f(a,b)=0\}.$$ For $A=\emptyset$ we denote $\tp(b/\emptyset)=\tp(b)$. A more model-theoretic presentation of types in continuous logic is given in \cite[\textsection 8]{bbhu08} or in \cite[\textsection 3]{benusv10}; there, a type is a set of \emph{conditions} which an element may eventually satisfy. A type $p$ given as an ideal is identified with the set of conditions of the form $h(y)=0$ for $h\in p$.

The space of types over $A$ (that is, the maximal ideal space of the algebra of $A$-definable predicates, with its natural topology) is denoted by $S^M_y(A)$, or by $S(A)$ when $\beta=1$ and the structure is clear from the context. If $A$ is $G$-invariant, then the algebra of $A$-definable predicates is $G$-invariant and there is a natural action of $G$ on $S^M_y(A)$. Thus, for example, the type space $S(M)$ (together with the natural map $\tp\colon M\to S(M)$) is just the compactification $M^{\DEF(M)}$. In particular, if $G$ is Roelcke precompact, then by Proposition \ref{Form=RUC}, Remark \ref{compactifications of G} and the discussion about the structure $\widehat{G}_L$ above, we have that $S(\widehat{G}_L)=G^{\UC}$ is just the Roelcke compactification of $G$.

\begin{rem}\label{X_f} Let $f=f(x,y)$ be an arbitrary formula and let $a\in M^\alpha$ be a parameter. The cyclic $G$-space of $f_a$ (as defined after Remark \ref{compactifications of G}) also has a name in the model-theoretic literature, at least for some authors: it coincides with the space of \emph{$f$-types} over the orbit $Ga$ as defined in \cite[p.~132]{tenzie}. Their definition is in the classical setting, but we can adapt it to the metric case by defining a (complete) $f$-type over $A\subset M^\alpha$ to be a maximal consistent set of conditions of the form $f(a',y)=r$ for $a'\in A$ and $r\in\RR$. In other words, an $f$-type is a maximal ideal of the closed unital algebra generated by $\{f_{a'}:a'\in A\}$. The space of $f$-types over $A$ is denoted by $S_f(A)$, and the identification $S_f(Ga)=X_{f_a}$ follows.

N.B. This does not coincide in general with the space $S_f(A)$ as defined in \cite{benusv10}, Definition~6.6, or in \cite[p.~14]{pil96}. To make the comparison simpler, say $A=B^\alpha$ for some $B\subset M$. The two definitions agree when $B=M$. In the case $B\subset M$, the latter authors define $S_f(A)$ (or $S_f(B)$ in their notation) as the maximal ideal space of the algebra of $B$-definable predicates in $M$ that come from the compactification $S_f(M)$. This is larger than the one defined above, and it fits better for the study of local stability.

We shall understand $S_f(A)$ in the former sense (except in Lemma \ref{stable implies definable extensions}).\end{rem}

A tuple $b\in M^\beta$ \emph{realizes} a type $p\in S^M_y(A)$ if we have $\tp(b/A)=p$. A set $q$ of $M$-definable predicates in the variable $y$ is \emph{approximately finitely realized in $B\subset M$} if for every $\epsilon>0$ and every finite set of predicates $f_i\in q$, $i<k$, there is $b\in B^\beta$ such that $|f_i(b)|<\epsilon$ for each $i<k$. Remark that any $p\in S^M_y(M)$ is approximately finitely realized in $M$: if for example $f_a\in p$ is bounded away from zero in $M$, then $1/f_a$ is an $A$-definable predicate, hence $1=1/f_a\cdot f_a\in p$ and $p$ is not a proper ideal. Conversely, by Zorn's Lemma, any set of $A$-predicates in $y$ approximately finitely realized in $M$ can be extended to a type $p\in S^M_y(A)$.

The following terminology is not standard, so we single it out.

\begin{defin} We will say that a structure $M$ is \emph{$\emptyset$-saturated} if every type $p\in S^M_y(\emptyset)$ in any countable variable $y$ is realized in $M$.\end{defin}

Suppose $M$ is $\aleph_0$-categorical. Then the projection $M^\beta\to M^\beta\sslash G$ is a compactification, and the functions that come from it are precisely the continuous $G$-invariant ones, i.e.\ the $\emptyset$-definable predicates. Hence the projection to $M^\beta\sslash G$ can be identified with the compactification $\tp\colon M^\beta\to S^M_y(\emptyset)$. A first consequence of this identification is the following homogeneity property: if $\tp(a)=\tp(b)$ for $a,b\in M^\beta$ and we have $\epsilon>0$, then there is $g\in G$ with $d(a,gb)<\epsilon$. A further consequence is the following.

\begin{fact} Every $\aleph_0$-categorical structure is $\emptyset$-saturated.\end{fact}

A stronger saturation property is true for $\aleph_0$-categorical structures (they are \emph{approximately $\aleph_0$-saturated}, see Definition 1.3 in \cite{benusv-d-finiteness}), but we will not use it.

\begin{rem}\label{GL is not saturated} The left completion $M=\widehat{G}_L$, when seen as a metric structure as defined before, is $\emptyset$-saturated if and only if it is $\aleph_0$-categorical. Indeed, if it is not $\aleph_0$-categorical then the quotient $M^n\sslash G$ is not compact for some $n<\omega$, which means that there are $\epsilon>0$ and a sequence of orbits $(i_k)_{k<\omega}\subset M^n\sslash G$ any two of which are at distance at least $\epsilon$. We may moreover assume that $(i_k)_{k<\omega}$ is maximal such, since $M^n$ is separable. If, as before, $P_i\colon M^n\to\RR$ denotes the distance to the orbit $i\in M^n\sslash G$, then the conditions $\{P_{i_k}(y)\geq\epsilon\}_{k<\omega}$ induce a type over $\emptyset$ not realized in $M$.\end{rem}

Now suppose that we have a metric structure $M$ given by the basic predicates $f_i\colon M^{n_i}\to\RR$, $i\in I$. An \emph{elementary extension} of $M$ is a structure $N$ with basic predicates $\tilde{f_i}\colon N^{n_i}\to\RR$, $i\in I$, such that: (i) $M$ is a metric subspace of $N$, (ii) each $\tilde{f_i}$ extends $f_i$, and (iii) every type $p\in S^N_y(M)$ is approximately finitely realized in $M\subset N$. One can deduce that the $M$-definable predicates of $M$ are exactly the restrictions to $M$ of the $M$-definable predicates of $N$ (essentially, because (iii) ensures that approximate quantification over $M$ and over $N$ coincide), and the restriction is one-to-one. Hence, the spaces $S^M_y(M)$ and $S^N_y(M)$ can be identified. A metric ultrapower construction as in \cite[\textsection 5]{bbhu08} can be used to prove the following.

\begin{fact} Every metric structure $M$ admits an elementary extension $N$ such that every type in $S_y(M)$ in any countable variable $y$ is realized in $N$. (In particular, every structure has a $\emptyset$-saturated elementary extension.)\end{fact}

Thus, for most purposes, we can refer to types over $M$ or to elements in elementary extensions of $M$ interchangeably. For example, if $p\in S(M)$, $f_a\in\DEF(M)$, and $b$ is an element in an elementary extension $N$ of $M$ realizing $p$, we may prefer to write $f(a,b)$ instead of $\tilde{f}_a(p)$. We recall that the formula $f(x,y)$ of $M$ extends uniquely to a formula of $N$, and we identify them.

An \emph{indiscernible sequence} in a structure $M$ is a sequence $(a_i)_{i<\omega}\subset M^\beta$ such that, for any $i_1<\dots<i_k<\omega$, we have $\tp(a_{i_1}\dots a_{i_k})=\tp(a_1\dots a_k)$. In a finitary version, if $\Delta$ is a finte set of definable predicates and $\delta$ is a positive real, then a sequence $(a_i)_{i<\omega}$ is \emph{$\Delta$-$\delta$-indiscernible} if $|\phi(a_{i_1},\dots,a_{i_k})-\phi(a_{j_1},\dots,a_{j_k})|\leq\delta$ for every $i_1<\dots<i_k$, $j_1<\dots<j_k$ and every definable predicate $\phi(y_1,\dots,y_k)\in\Delta$.

Finally, we shall say that a subset $A\subset M^\alpha$ is \emph{type-definable} if it is of the form $\{a\in M^\alpha:f_j(a)=0\text{ for all }j\in J\}$ for a family of definable predicates $f_j(x)$, $j\in J$. In particular, every type-definable set is $G$-invariant and closed. If $M$ is $\aleph_0$-categorical, then any $G$-invariant closed set is type-definable, even by a single predicate, namely the (continuous $G$-invariant) distance function $P_A(x)=d(x,A)$. In general, $A$ is called \emph{definable} precisely when the distance function $P_A$ is a definable predicate. In the latter case, if $f(x,y)$ is any formula, then $F(y)=\sup_{x\in A}f(x,y)$ is a definable predicate too, and similarly for the infimum (see \cite{bbhu08}, Theorem 9.17). If $A$ is definable and $N$ is an elementary extension of $M$, $P_A$ will denote the definable predicate that coincides with $d(x,A)$ on $M^\alpha$; thus an element $a\in N^\alpha$ satisfying $P_A(a)=0$ need not be in $A$.

\subsection{Almost periodic functions} We end this section with some comments about the smallest function algebra presented in the introduction.
A continuous bounded function $h$ on a metric $G$-space $X$ is \emph{almost periodic ($\AP$)} if the orbit $Gh$ is a precompact subset of $\CC(X)$ (with respect to the topology of the norm). As is easy to check, the family $\AP(X)$ of almost periodic functions on $X$ is a uniformly closed $G$-invariant subalgebra of $\RUC(X)$. Moreover, if $h$ comes from a compactification $\nu\colon X\to Y$, it is clear that $h$ is $\AP$ if and only if its extension to $Y$ is $\AP$. By the Arzel\`a--Ascoli theorem, we have that $h\in\AP(Y)$ if and only if for every $\epsilon>0$ and $y\in Y$ there is an open neighborhood $O$ of $y$ such that $$|h(gy)-h(gy')|<\epsilon$$ for every $y'\in O$ and $g\in G$. From the point of view of Banach space representations, almost periodic functions are precisely those coming from Euclidean-approximable compactifications of $X$; see \cite[\textsection 5.2]{glameg13} and \cite{meg08}, Proposition 3.7.2.

The definition given in the following proposition will be useful for the description of $\AP$ functions in the examples of Section \ref{Examples}. (The terminology is not standard.)

\begin{prop}\label{AP formulas} Let $M$ be a $\emptyset$-saturated structure. Let $f(x,y)$ be a formula and $A\subset M^\alpha$, $B\subset M^\beta$ be definable sets. The following are equivalent, and in any of these cases we will say that $f(x,y)$ is \emph{algebraic on $A\times B$}.
\begin{enumerate}
\item the set $\{f_a|_B:a\in A\}$ is precompact in $\CC(B)$;
\item for every indiscernible sequence $(a_i)_{i<\omega}\subset A$, the predicates $f(a_i,y)$ are all equivalent in $B$, i.e.\ we have $f(a_i,b)=f(a_j,b)$ for all $i,j$ and $b\in B$.
\end{enumerate}
\end{prop}
\begin{proof}
$(1)\Rightarrow(2)$. By precompactness, the sequence $(f_{a_i})_{i<\omega}$ has a Cauchy subsequence, so in particular there are $i$ and $j$ such that $\sup_{y\in B}|f(a_i,y)-f(a_j,y)|\leq\epsilon$. By indiscernibility, this is true for all $i,j$, and the claim follows.

$(2)\Rightarrow(1)$. Let $\epsilon>0$. If the set of conditions in the variables $(x_i)_{i<\omega}$ given by $$|\phi(x_{i_1},\dots,x_{i_k})-\phi(x_{j_1},\dots,x_{j_k})|=0,\ P_A(x_i)=0,\ \sup_{y\in B}|f(x_i,y)-f(x_j,y)|\geq\epsilon$$ (where $\phi$ varies over the definable predicates of $M$, $i_1<\dots<i_k$, $j_1<\dots<j_k$), was approximately finitely realized in $M$, then by $\emptyset$-saturation we could get an indiscernible sequence in $M$ contradicting (2). Therefore, there are a finite set $\Delta$ of definable predicates and $\delta>0$ such that any $\Delta$-$\delta$-indiscernible sequence $(a_i)_{i<\omega}\subset A$ satisfies $\sup_{y\in B}|f(a_i,y)-f(a_j,y)|<\epsilon$ for all $i,j$.

For every $n<\omega$ let $\Delta_n$, $\delta_n$ correspond to $\epsilon=1/n$ as before. Starting with an arbitrary sequence $(a_i)_{i<\omega}\subset A$, by Ramsey's theorem we can extract a $\Delta_1$-$\delta_1$-indiscernible subsequence, say $(a^1_i)_{i<\omega}$. Inductively, let $(a^{n+1}_i)_{i<\omega}$ be a $\Delta_{n+1}$-$\delta_{n+1}$-indiscernible subsequence of $(a^n_i)_{i<\omega}$. If we take $a^\omega_j=a_j^j$ then $(a^\omega_j)_{j<\omega}$ is a subsequence of $(a_i)_{i<\omega}$ and $(f_{a^\omega_j})_{j<\omega}$ is a Cauchy sequence in $\CC(B)$.
\end{proof}

\begin{rem} As the reader can check, the previous proposition holds true if $A$ and $B$ are merely type-definable. In particular, one may consider the case where $A=\{a'\in M^\alpha:\tp(a')=\tp(a)\}$ for some $a\in M^\alpha$, and $B=M^\beta$. If the above equivalent conditions hold in this case (for a saturated model $M$), it is standard terminology to say that (the \emph{canonical parameter} of) $f_a$ is \emph{algebraic over the empty set}, in symbols $f_a\in\acl(\emptyset)$. Alternatively, in the terminology of Pillay \cite[p.~9]{pil96}, $f_a$ is \emph{almost $\emptyset$-definable}.

Set $\beta=1$. Suppose $M$ is $\aleph_0$-categorical, so in particular $A=[a]=\overline{Ga}$. Now, since $f(x,y)$ is uniformly continuous, the families $Gf_a$ and $\{f_{a'}:a'\in [a]\}$ have the same closure in $\CC(M)$. We can conclude by Proposition \ref{Form=RUC} that $h\in\AP_u(M)$ if and only if $h=f_a$ for some predicate $f_a\in\acl(\emptyset)$.\end{rem}

The compactification $b\colon G\to bG=G^{\AP}$ associated to the algebra $\AP(G)$ is the \emph{Bohr compactification} of~$G$. The space $bG$ has the structure of a (compact) group making $b$ a homomorphism (see \cite[(D.12)3 and IV(6.15)3]{vriesElements}). In fact, the compactification $b$ is the \emph{universal group compactification} of $G$: if $\nu\colon G\to K$ is a compactification and also a homomorphism into a compact group $K$, it is easy to see that $\AA_\nu\subset\AP(G)$, whence $\nu$ factors through $b$. I. Ben Yaacov has observed the following fact.

\begin{theorem} The Bohr compactification $b\colon G\to bG$ of a Roelcke precompact Polish group is always surjective.\end{theorem}

See \cite{benBohr}, Theorem 5.5. As mentioned there in the introduction, the model-theoretic counterpart of this result is the fact that $\aleph_0$-categoricity is preserved after naming the algebraic closure of the empty~set (see Proposition 1.15).

One could call a metric $G$-space $X$ \emph{almost periodic} if $\AP(X)=\RUC_u(X)$. This is a very strong condition. Indeed, for an action of a topological group $G$ by isometries on a complete bounded metric space $(X,d)$, the function $P_a(y)=d(a,y)$ (which is in $\RUC_u(X)$) is $\AP$ if and only if the closed orbit $[a]$ is compact. If the space of closed orbits $X\sslash G$ is compact, we can deduce that $X$ is almost periodic if and only if $X$ is compact (the reverse implication following from Arzel\`a--Ascoli theorem). This is the case for $\aleph_0$-categorical structures. Also, if $G$ is any Polish group with $\AP(G)=\UC(G)$ then $b\colon G\to bG$ is a topological embedding into a compact Hausdorff group, which implies that $G$ is already compact (see \cite[D.12.4]{vriesElements} together with \cite[p.~3--4]{kechrisPolish}).

\noindent\hrulefill
\section{$\WAP = \Asp = \SUC$}\label{wap=suc}

Let $f\colon M^\alpha\times M^\beta\to\RR$ be any formula on a metric structure $M$, and let $A\subset M^\alpha$, $B\subset M^\beta$ be any subsets. We recall that $f(x,y)$ has the \emph{order property}, let us say, \emph{on $A\times B$} if there are $\epsilon>0$ and sequences $(a_i)_{i<\omega}\subset A$, $(b_j)_{j<\omega}\subset B$ such that $|f(a_i,b_j)-f(a_j,b_i)|\geq\epsilon$ for all $i<j<\omega$. If $f(x,y)$ lacks the order property on $A\times B$ we say that it is \emph{stable on $A\times B$}. We invoke the following crucial result, essentially due to Grothendieck, as pointed out by Ben Yaacov in \cite{ben13} (see Fact 2 and the discussion before Theorem 3 therein).

\begin{fact}\label{stable is WAP} The formula $f(x,y)$ is stable on $A\times B$ if and only if $\{f_a|_B:a\in A\}$ is weakly precompact in~$\CC(B)$.\end{fact}

\noindent\emph{(In the rest of this section we will  only need the case $B=M$ ($\beta=1$), so we shall only specify $A$ when referring to stability or the order property.)}

\medskip

On the other hand, a function $h\in\CC(X)$ on a $G$-space $X$ is \emph{weakly almost periodic ($\WAP$)} if the orbit $Gh\subset\CC(X)$ is weakly precompact (that is, precompact with respect to the weak topology on $\CC(X)$). It is not difficult to check that the family $\WAP(X)$ of weakly almost periodic functions on $X$ is a uniformly closed $G$-invariant subalgebra of $\CC(X)$ (for instance, resorting to Grothendieck's double limit criterion: Fact~2 in \cite{ben13}), but it is a bit involved to prove that $\WAP(X)$ is in fact a subalgebra of $\RUC(X)$; see Fact~2.7 in \cite{meg03} and the references thereof. If one knows that a function $h\in\CC(X)$ comes from a compactification $\nu\colon X\to Y$, it is an immediate consequence of Grothendieck's double limit criterion (in the form stated in \cite{ben13}) that $h\in\WAP(X)$ if and only if $\tilde{h}\in\WAP(Y)$.

From Fact \ref{stable is WAP} above we have, for $M$-definable predicates, that $f_a$ is $\WAP$ if and only if $f(x,y)$ is stable on $A=Ga$ (equivalently, on its closure $[a]$). By Proposition \ref{Form=RUC} one concludes the following (compare with Lemma 5.1 in \cite{bentsa}).

\begin{lem} If $M$ is $\aleph_0$-categorical, then a continuous function is in $\WAP_u(M)$ if and only if it is of the form $f_a$ for a formula $f(x,y)$ stable on $[a]$.\end{lem}

The algebra $\WAP(X)$ can also be characterized as the class of functions coming from a reflexive-representable compactification of $X$. This was first proven in \cite{meg03}, Theorem~4.6; for an alternative exposition see Theorem 2.9 in \cite{meg08}. (We also point out the paper of Iovino \cite{iov98} for an earlier treatment of the connection between stability and reflexive Banach spaces.)

A natural generalization of weak almost periodicity is thus to replace \emph{reflexive} by \emph{Asplund} in the latter characterization. Recall that a Banach space is Asplund if the dual of every separable subspace is separable, and that every reflexive space has this property. In this way one gets the family of \emph{Asplund functions}, $\Asp(X)$. This is a uniformly closed $G$-invariant subalgebra of $\RUC(X)$. See \cite[\textsection 7]{meg03}.

For a \emph{compact} $G$-space $Y$, a function $h\in\CC(Y)$ is shown to be Asplund if and only if the orbit $Gh\subset\CC(Y)$ is a \emph{fragmented family}; see Theorem 9.12 in \cite{glameg06}. This means that for any nonempty $B\subset Y$ and any $\epsilon>0$ there exists an open set $O\subset Y$ such that $B\cap O$ is nonempty and $$|h(gy)-h(gy')|<\epsilon$$ for every $g\in G$ and $y,y'\in B\cap O$. If $X$ is an arbitrary $G$-space, then $h\in\CC(X)$ belongs to $\Asp(X)$ if and only if it comes from an Asplund function on some compactification of~$X$. If a function $h$ comes from two compactifications $Y$ and $Z$ with $Y$ larger than $Z$, it is an exercise (using the characterization by fragmentability) to check that the extension of $h$ to $Y$ is Asplund if and only if so is its extension to $Z$ (see the proof of Lemma 6.4 in \cite{glameg06}). That is, any extension of $h$ to some compactification can be used to check whether $h$ is Asplund; for example, a predicate $f_a\in\DEF(M)$ is Asplund if and only if its extension to $S(M)$ or to $S_f(Ga)$ satisfies the fragmentability condition.

It will be interesting to bring in a further weaker notion, introduced in \cite{glameg08}. A function $h\in\CC(Y)$ on a compact $G$-space is \emph{strongly uniformly continuous ($\SUC$)} if for every $y\in Y$ and $\epsilon>0$ there exists a neighborhood $U$ of the identity of $G$ such that $$|h(gy)-h(guy)|<\epsilon$$ for all $g\in G$ and $u\in U$. In this case it is immediate that, if $j\colon Y\to Z$ is a compactification between compact $G$-spaces, then $h\in\CC(Z)$ is $\SUC$ if and only if $hj\in\CC(Y)$ is $\SUC$. A function $h\in\RUC(X)$ on an arbitrary $G$-space $X$ is called $\SUC$ if its extension to some (any) compactification (from which $h$ comes) is $\SUC$. One can see readily that: (i) the family of all strongly uniformly continuous functions on a $G$-space $X$ forms a uniformly closed $G$-invariant subalgebra $\SUC(X)$ of $\RUC(X)$; (ii) every Asplund function is $\SUC$: in the fragmentability condition we take $B=Gy\subset Y$, then use the continuity of the action of $G$ on $Y$.

It follows from our remarks so far that, in general, $$\WAP(X)\subset\Asp(X)\subset\SUC(X)\subset\RUC(X).$$ It is also clear that $\SUC(G)\subset\UC(G)$ for the regular left action of $G$ on itself: we apply the property defining $\SUC$ to the compactification $Y=G^{\RUC}$ (for instance) and the identity element $y=1\in G\subset Y$.

An important motivation for the algebra of $\SUC$ functions comes from the viewpoint of semigroup compactifications of $G$. We have already mentioned the universal property of $G^{\AP}$. Similarly, $G^{\WAP}$ is the \emph{universal semitopological semigroup compactification} of $G$ (see \cite[\textsection 5]{uspComp}). For their part, $G^{\Asp}$ and $G^{\RUC}$ are \emph{right topological semigroup compactifications} of $G$. In their work \cite{glameg08}, the authors showed that the compactification $G^{\SUC}$ is also a right topological semigroup compactification of $G$, and that $\SUC(G)$ is the largest subalgebra of $\UC(G)$ with this property (see Theorem 4.8 therein). In particular, the Roelcke compactification $G^{\UC}$ has the structure of a right topological semigroup if and only if $\SUC(G)=\UC(G)$.

We aim to prove the equality $\WAP=\SUC$ (restricted to $\RUC_u$) for $\aleph_0$-categorical structures and for their automorphism groups.

Switching to logic language, let us say that a formula $f(x,y)$ is \emph{$\SUC$ on a subset $A\subset M^\alpha$} if for any $b$ in any elementary extension of $M$ and every $\epsilon>0$ there are $\delta>0$ and a finite tuple $c$ from $M$ such that for every $a\in A$ and every automorphism $u\in G$ satisfying $d(uc,c)<\delta$ we have $$|f(a,b)-f(ua,b)|<\epsilon.$$ We readily get the following.

\begin{lem}\label{SUC formulas} For a metric structure $M$, a function $f_a\in\DEF(M)$ is $\SUC$ if and only if the formula $f(x,y)$ is $\SUC$ on $[a]$.\end{lem}

The most basic example of a non-stable formula in an $\aleph_0$-categorical structure is the order relation on the countable dense linear order without endpoints. It is worth looking into this case.

\begin{example}\label{ejemplo de Q} The order relation $x<y$ on the (classical) structure $(\QQ,<)$ is not $\SUC$ on $\QQ$. Indeed, let $r\in\RR\setminus\QQ$, $c_1,\dots,c_n\in\QQ$. Suppose $c_i<c_{i+1}$ for each $i$, and say $c_{i_0}<r<c_{i_0+1}$. Take $a\in\QQ$, $c_{i_0}<a<r$. There is a monotone bijection $u$ fixing every $c_i$ and such that $r<ua<c_{i_0+1}$. The claim follows.\end{example}

We will generalize the analysis of this simple example to any non-stable formula in any $\aleph_0$-categorical structure. To this end we shall use the following standard lemma.

\begin{fact}\label{sequence indexed by Q} Let $M$ be $\emptyset$-saturated, $A\subset M^\omega$ a type-definable subset. If a formula $f(x,y)$ has the order property on $A$, then there are an elementary extension $N$ of $M$, distinct real numbers $r,s\in\RR$ and elements $(a_i)_{i\in\QQ}\subset A$, $(b_j)_{j\in\RR}\subset N$ such that $f(a_i,b_j)=r$ for $i<j$ and $f(a_i,b_j)=s$ for $j\leq i$.\end{fact}
\begin{proof} Suppose there are $\epsilon>0$ and sequences $(a'_k)_{k<\omega}\subset A$, $(b'_l)_{l<\omega}\subset M$ such that $|f(a'_k,b'_l)-f(a'_l,b'_k)|\geq\epsilon$ for all $k<l<\omega$. Since $f$ is bounded, passing to subsequences carefully we can assume that $\lim_k\lim_lf(a'_k,b'_l)=r$ and $\lim_k\lim_lf(a'_l,b'_k)=s$, necessarily with $|r-s|\geq\epsilon>0$. Now we consider the conditions in the countable variables $(x_i)_{i\in\QQ}$, $(y_j)_{j\in\QQ}$ asserting, for each pair of rational numbers $i<j$, $$x_i\in A,\ f(x_i,y_j)=r\text{ and }f(x_j,y_i)=s.$$ The elements $a'_k,b'_l$ can be used to show that these conditions are approximately finitely realized in $M$. By saturation, there are $(a_i)_{i\in\QQ}\subset A$, $(b_j)_{j\in\QQ}\subset M$ satisfying the conditions.

Finally, the conditions $f(a_i,y_j)=r$ for $i<j$, $i\in\QQ$, $j\in\RR\setminus\QQ$, together with $f(a_i,y_j)=s$ for $j\leq i$, $i\in\QQ$, $j\in\RR\setminus\QQ$, are approximately finitely realized in $\{b_j\}_{j\in\QQ}$. Hence they are realized by elements $(b_j)_{j\in\RR\setminus\QQ}$ in some elementary extension of $M$.\end{proof}

\begin{prop}\label{mainprop} Let $M$ be $\aleph_0$-categorical. If $f(x,y)$ has the order property on a definable set $A$, then $f(x,y)$ is not $\SUC$ on $A$.\end{prop}
\begin{proof} We apply the previous fact to find elements $(a_i)_{i\in\QQ}\subset A$, $(b_j)_{j\in\RR}$ in some elementary extension of $M$ and real numbers $r\neq s$ such that $f(a_i,b_j)=r$ if $i<j$, $f(a_i,b_j)=s$ if $j\leq i$. Suppose $f(x,y)$ has the $\SUC$\ property for $\epsilon=|r-s|/2$; since $G$ is second countable and $\RR$ is uncountable, there is an open neighborhood $U$ of the identity that witnesses the property for an infinite number of elements $b_j$, say for every $b_j$ with $j$ in an infinite set $J\subset\RR$. By passing to a subset we may assume that $J$ is discrete, and thus for each $j\in J$ we may take a rational $i(j)<j$ such that $j'<i(j)$ for every $j'<j$, $j'\in J$. We may assume that $U$ is the family of automorphisms moving a finite tuple $c$ at a distance less than $\delta$; say $n$ is the length of the tuple $c$. Now let $\eta=\Delta_f(|r-s|/2)$, where $\Delta_f$ is a modulus of uniform continuity for~$f(x,y)$.

Since $M$ is $\aleph_0$-categorical the quotient $M^\omega\sslash G$ is compact, so there must be a pair $j<j'$ in $J$ and an automorphism $u$ such that 
$$d(u(ca_{i(j)}),ca_{i(j')})<\min(\delta,\eta/2^n).$$ In particular $d(uc,c)<\delta$, so $u\in U$. In addition, since $d(ua_{i(j)},a_{i(j')})<\eta$, $f(a_{i(j')},b_j)=s$ and $f(a_{i(j)},b_j)=r$, we have $$|f(a_{i(j)},b_j)-f(ua_{i(j)},b_j)|\geq |r-s|/2,$$ contradicting the fact that $U$ witnesses the $\SUC$\ property for $b_j$ and $\epsilon=|r-s|/2$.\end{proof}

\begin{rem} We can offer a maybe more conceptual argument to a model-theorist. Suppose $M$ is $\aleph_0$-categorical, take $f_a\in\SUC_u(M)\subset\DEF(M)$ (we recall Proposition \ref{Form=RUC}) and let $p$ be a type in $S_f(Ga)$. Consider $d_pf\colon Ga\to\RR$ given by $d_pf(ga)=g\tilde{f_a}(p)$, which is well-defined and uniformly continuous. Now, the $\SUC$ condition for the extension $\tilde{f_a}\colon S_f(Ga)\to\RR$ gives, for every $\epsilon>0$, a neighborhood $U$ of the identity of $G$ such that $$|d_pf(u^{-1}ga)-d_pf(ga)|<\epsilon$$ for every $g\in G$ and $u\in U$. That is to say, $d_pf\in\RUC_u(Ga)$. A mild adaptation of Proposition~\ref{Form=RUC} allows us to deduce that $d_pf$ is an $M$-definable predicate on $Ga$. In other words, \emph{every $f$-type over $Ga$ is definable in $M$}, which (bearing in mind that $M$ is saturated and that $[a]$ is definable) is well-known to be equivalent to the stability of $f(x,y)$ on $Ga$.
For more on definability of types in continuous logic see \cite[\textsection 7]{benusv10} (particularly Proposition 7.7 for the equivalences of stability), and the topical discussion of \cite{ben13}. Yet an argument based on some variation of Fact \ref{sequence indexed by Q} is needed to prove that definability of types implies stability.
\end{rem}

The proposition and previous lemmas yield the desired conclusion.

\begin{cor} Let $M$ be an $\aleph_0$-categorical structure. Then $\WAP_u(M)=\Asp_u(M)=\SUC_u(M)$.\end{cor}

\begin{theorem}\label{main} Let $G$ be a Roelcke precompact Polish group. Then $\WAP(G)=\Asp(G)=\SUC(G)$.\end{theorem}
\begin{proof} From Remark \ref{compactifications of G} we can deduce that the isomorphism $\RUC_u(\widehat{G}_L)\simeq\UC(G)$ preserves $\WAP$ and $\SUC$ functions. Thus if $f\in\SUC(G)$ then its continuous extension $\tilde{f}$ to $\widehat{G}_L$ is $\SUC$, so by the previous corollary $\tilde{f}\in\WAP(\widehat{G}_L)$; hence $f\in\WAP(G)$.
\end{proof}

For the case of Asplund functions we can give a slight generalization, which applies for example to any $M$-definable predicate in an approximately $\aleph_0$-saturated separable structure.

If $c$ is an $n$-tuple (of tuples) and $I$ a is an $n$-tuple of intervals of $\RR$, let us write $f(c,d)\in I$ instead of $(f(c_k,d))_{k<n}\in\prod_{k<n}I_k$. Let us call a formula $f(x,y)$ \emph{Asplund on a subset $A\subset M^\alpha$} of a metric structure $M$ if it \emph{lacks} the following property: (SP) There exist $\epsilon>0$ and a set $B$ in some elementary extension of $M$ such that, if $f(c,d)\in I$ for some $d\in B$, some tuple $c$ from $A$ and some tuple $I$ of open intervals of $\RR$, then there are $b,b'\in B$, $a\in A$ with $f(c,b),f(c,b')\in I$ and $|f(a,b)-f(a,b')|\geq\epsilon$. This makes a function $f_a\in\DEF(M)$ Asplund in the topological sense if and only if $f(x,y)$ is Asplund on the orbit $Ga$, or on its closure $[a]$.

\begin{prop}\label{Asp prop} Let $M$ be a separable $\emptyset$-saturated structure. Let $f(x,y)$ be a formula and $a\in M^\alpha$ a parameter, and suppose that the closed orbit $[a]$ is type-definable. If $f_a\in\Asp(M)$, then $f_a\in\WAP(M)$.
\end{prop}
\begin{proof} Suppose $f(x,y)$ has the order property on $[a]$. Let $(a_i)_{i\in\QQ}\subset [a]$, $(b_j)_{j\in\RR}$ and $r,s\in\RR$ be as given by Fact \ref{sequence indexed by Q}. Since $M$ is separable, it is enough to check the condition SP for a countable family $C$ of pairs $(c,I)$. There is at most a countable number of reals $l$ such that, for some $(c,I)\in C$, we have $f(c,b_j)\in I$ if and only if $j=l$. So by throwing them away we may assume that, whenever $f(c,b_l)\in I$, $(c,I)\in C$, there is $j\neq l$ with $f(c,b_j)\in I$; if we then choose $i\in\QQ$ lying between $l$ and $j$, we have $|f(a_i,b_l)-f(a_i,b_j)|=|r-s|$. Hence $f(x,y)$ has SP for $\epsilon=|r-s|$ and $B=\{b_j\}$.\end{proof}

The previous proposition can be used to get information about certain continuous functions on some (non Roelcke precompact) Polish groups, but not via the structure $M=\widehat{G_L}$, which in general is not $\emptyset$-saturated as mentioned in Remark \ref{GL is not saturated}. Instead, it may be applied to automorphism groups of saturated structures and functions of the form $g\mapsto f(a,gb)$.

\begin{example} Let us consider the linearly ordered set $M=(\ZZ,<)$ (which, as a $G$-space, can be identified with its automorphism group, $G=\ZZ$). The basic predicate is given by $P_<(x,y)=0$ if $x<y$, $P_<(x,y)=1$ otherwise. The indicator function of the non-positive integers, $f=\mathds{1}_{\ZZ_{\leq 0}}\in\CC(\ZZ)$, is an $M$-definable predicate, $f(y)=P_<(0,y)$. It is clearly not in $\WAP(\ZZ)$. However, it comes from the two-point compactification $X=\ZZ\cup\{-\infty,+\infty\}$, and it is easy to check that its extension to $X$ satisfies the fragmentability condition, whence in fact $f\in\Asp(\ZZ)$ (more generally, see \cite{glameg06}, Corollary 10.2). Of course, $M$ is not $\emptyset$-saturated.

On the other hand, we can consider the linearly ordered set $N=\bigsqcup_{i\in(\QQ,<)}(\ZZ,<)_i$ (where each $(\ZZ,<)_i$ is a copy of $(\ZZ,<)$), which is a $\emptyset$-saturated elementary extension of $M$ (say $M=(\ZZ,<)_0$). The automorphism group of $N$ is $G=\ZZ^\QQ\rtimes\Aut(\QQ,<)$. As an $M$-definable predicate on $N$, $f$~is the indicator function of the set of elements of $N$ that are not greater than $0\in M\subset N$. As before, $f\notin\WAP(N)$, but then by Proposition \ref{Asp prop} we have $f\notin\Asp(N)$ either (note that the orbit of $0$ is~$N$). As per Proposition \ref{Form=UC}, the function $h\colon g\mapsto f(g(0))$ is in $\UC(G)$. Since the continuous $G$-map $g\in G\mapsto g(0)\in N$ is surjective, any compactification of $N$ induces a compactification of $G$. It follows that $h\in\UC(G)\setminus\Asp(G)$. (However, here one can also adapt the argument of Example~\ref{ejemplo de Q} to show that in fact~$f\notin\SUC(N)$ and hence $h\notin\SUC(G)$.)\end{example}

\noindent\hrulefill
\section{$\Tame\cap\UC = \NIP = \Null\cap\UC$}\label{tame=nip}

Tame functions have been studied by Glasner and Megrelishvili in \cite{glameg12}, after the introduction of tame dynamical systems by K\"ohler \cite{koh95} (who called them \emph{regular systems}) and later by Glasner in \cite{gla06}. If the translation of Ben Yaacov and Tsankov for Roelcke precompact Polish groups identifies $\WAP$ functions with stable formulas, we remark in this section that tame functions correspond to \emph{$\NIP$} (or \emph{dependent}) formulas. The study of this model-theoretic notion, a generalization of local stability introduced by Shelah \cite{she71}, is an active and important domain of research, mainly in the classical first-order setting ---though, as the third item of the following proposition points out, the notion has a very natural metric presentation.

\begin{prop}\label{NIP} Let $M$ be a $\emptyset$-saturated structure. Let $f(x,y)$ be a formula and $A\subset M^\alpha$, $B\subset M^\beta$ be definable sets. The following are equivalent; in any of these cases, we will say that $f(x,y)$ is $\NIP$ on $A\times B$.
\begin{enumerate}
\item There do not exist real numbers $r\neq s$, a sequence $(a_i)_{i<\omega}\subset A$ and a family $(b_I)_{I\subset\omega}$ in some elementary extension, with $P_B(b_I)=0$ for all $I\subset\omega$, such that for all $i<\omega$, $I\subset\omega$, $$f(a_i,b_I)=r\text{ if }i\in I\text{ and }f(a_i,b_I)=s\text{ if }i\notin I.$$
\item For every indiscernible sequence $(a_i)_{i<\omega}\subset A$ and every $b\in B$ (equivalently, for every $b$ in any elementary extension satisfying $P_B(b)=0$), the sequence $(f(a_i,b))_{i<\omega}$ converges in~$\RR$.
\item Every sequence $(a_i)_{i<\omega}\subset A$ admits a subsequence $(a_{i_j})_{j<\omega}$ such that $(f(a_{i_j},b))_{j<\omega}$ converges in~$\RR$ for any $b$ in any elementary extension satisfying $P_B(b)=0$.
\end{enumerate}
\end{prop}
\begin{proof} $(1)\Rightarrow(2)$. Let $(a_i)_{i<\omega}\subset A$ be indiscernible, $b$ arbitrary with $P_B(b)=0$. If $(f(a_i,b))_{i<\omega}$ does not converge, there exist reals $r\neq s$ such that (replacing $(a_i)_{i<\omega}$ by a subsequence) $f(a_{2i},b)\to r$, $f(a_{2i+1},b)\to s$. By $\emptyset$-saturation we may assume that $f(a_{2i},b)=r$ and $f(a_{2i+1},b)=s$ for all $i$. Given $I\subset\omega$, take a strictly increasing function $\tau\colon \omega\to\omega$ such that $\tau(i)$ is even if and only if $i$ is in $I$. By indiscernibility, the set of conditions $$\{f(a_i,y)=t:t\in\{r,s\},\ f(a_{\tau(i)},b)=t\},\ P_B(y)=0,$$ is approximately finitely realized in $M$; take $b_I$ to be a realization in some model. Thus, for all $i$ and $I$, $f(a_i,b_I)=r$ if $i\in I$ and $f(a_i,b_I)=s$ if $i\notin I$, contradicting $(1)$.

$(2)\Rightarrow(3)$. We claim that for every $\epsilon>0$ there are some $\delta>0$ and a finite set of formulas $\Delta$ such that, for any $b$ with $P_B(b)=0$ and every $\Delta$-$\delta$-indiscernible sequence $(a_i)_{i<\omega}\subset A$, there exists $N<\omega$ with $|f(a_i,b)-f(a_j,b)|<\epsilon$ for all $i,j\geq N$. Otherwise, there are $\epsilon>0$ and, for any $\Delta$, $\delta$ as before, a $\Delta$-$\delta$-indiscernible sequence $(a_i)_{i<\omega}\subset A$ and a tuple $b$ with $P_B(b)=0$ such that $|f(a_{2i},b)-f(a_{2i+1},b)|\geq\epsilon$ for all $i<\omega$. By $\emptyset$-saturation, we can assume that $(a_i)_{i<\omega}$ is indiscernible and $b\in B$. Then, by $(2)$, the sequence $(f(a_i,b))_{i<\omega}$ should converge, but cannot. The claim follows.

Now suppose that $\Delta_n$, $\delta_n$ correspond to $\epsilon=1/n$ as per the previous claim. Given any sequence $(a^n_i)_{i<\omega}$ we can extract, using Ramsey's theorem, a $\Delta_{n+1}$-$\delta_{n+1}$-indiscernible subsequence $(a^{n+1}_i)_{i<\omega}$. As in the proof of Proposition \ref{AP formulas}, starting with any $(a_i)_{i<\omega}=(a^0_i)_{i<\omega}$, proceeding inductively and taking the diagonal, we get a subsequence $(a_{i_j})_{j<\omega}$ such that $(f(a_{i_j},b))_{j<\omega}$ converges for any $b$ satisfying $P_B(b)=0$.

$(3)\Rightarrow(1)$. Assume we have $(a_i)_{i<\omega}$, $(b_I)_{I\subset\omega}$ and $r,s$ contradicting $(1)$.
If $(a_{i_j})_{j<\omega}$ is as given by $(3)$ and $J\subset\omega$ is infinite and coinfinite in $\{i_j:j<\omega\}$, then $(f(a_{i_j},b_J))_{j<\omega}$ converges to both $r$ and $s$, a contradiction.
\end{proof}

A subset of a topological space is said \emph{sequentially precompact} if every sequence of elements of the subset has a convergent subsequence; we can restate the third item of the previous proposition in the following manner.

\begin{cor}\label{NIPcorollary} Let $M$ be $\emptyset$-saturated and $A\subset M^\alpha$, $B\subset M^\beta$ be definable sets. A formula $f(x,y)$ is $\NIP$ on $A\times B$ if and only if $\{\tilde{f}_a|_{B^*}:a\in A\}$ is sequentially precompact in $\RR^{B^*}$, where $\tilde{f}_a|_{B^*}$ is the extension of $f_a$ to $B^*=\{p\in S_y(M):P_B\in p\}$. If $A'\subset A$ is a dense subset, it is enough to check that $\{\tilde{f}_a|_{B^*}:a\in A'\}$ is sequentially precompact in $\RR^{B^*}$.\end{cor}

The proposition and corollary hold true, with the same proof and the obvious adaptations regarding $P_B$, if $A$ and $B$ are merely type-definable. In the literature, a formula in a given theory is said simply \emph{$\NIP$} if the previous conditions are satisfied on $M^\alpha\times M^\beta$ for some saturated model $M$ of the theory.

We turn to the topological side. Tame dynamical systems were originally introduced in terms of the enveloping semigroup of a dynamical system, and admit several equivalent presentations. The common theme are certain dichotomy theorems that have their root in the fundamental result of Rosenthal \cite{ros74}: a Banach space either contains an isomorphic copy of $\ell_1$ or has the property that every bounded sequence has a weak-Cauchy subsequence.

A Banach space is thus called \emph{Rosenthal} if it contains no isomorphic copy of $\ell_1$. Then, a continuous function $f\in\CC(X)$ on an arbitrary $G$-space is \emph{tame} if it comes from a Rosenthal-representable compactification of $X$. See \cite{glameg12}, Definition 5.5 and Theorem 6.7. See also Lemma 5.4 therein and the reference after Definition 5.5 to the effect that the family $\Tame(X)$ of all tame functions on $X$ forms a uniformly closed $G$-invariant subalgebra of $\RUC(X)$. For metric $X$ we shall mainly consider the restriction $\Tame_u(X)=\Tame(X)\cap\RUC_u(X)$, as per Note \ref{metric setting}.

From Proposition 5.6 and Fact 4.3 from \cite{glameg12} we have the following characterization of tame functions on compact systems.

\begin{fact}\label{tame} A function $f\in C(Y)$ on a compact $G$-space $Y$ is tame if and only if every sequence of functions in the orbit $Gf$ admits a weak-Cauchy subsequence or, equivalently, if $Gf$ is sequentially precompact in $\RR^Y$.\end{fact}

\begin{rem}\label{Y to Z} A direct consequence of this characterization is the following property: if $j\colon Y\to Z$ is a compactification between compact $G$-spaces and $h\in\CC(Z)$, then $h\in\Tame(Z)$ if and only if $hj\in\Tame(Y)$, which says that a function on an arbitrary $G$-space $X$ is tame if and only if all (or any) of its extensions to compactifications of $X$ are tame. (We had already pointed out the same property for $\AP$, $\WAP$, Asplund and $\SUC$ functions.) In fact, observe that the property holds true if $j\colon Y\to Z$ is just a continuous $G$-map with dense image between arbitrary $G$-spaces, since in this case $j$ induces a compactification $j_h\colon Y_{hj}\to Z_h$ between the corresponding (compact) cyclic $G$-spaces.\end{rem}

The link with $\NIP$ formulas is then immediate.

\begin{prop}\label{tameNIP} Let $M$ be an $\aleph_0$-categorical structure. Then $h\in\Tame_u(M)$ if and only if $h=f_a$ for a formula $f(x,y)$ that is $\NIP$ on $[a]\times M$. More generally, if $f(x,y)$ is a formula, $a\in M^\alpha$, and $B\subset M^\beta$ is definable, we have $f_a|_B\in\Tame(B)$ if and only if $f(x,y)$ is $\NIP$ on $[a]\times B$.
\end{prop}
\begin{proof} The first claim follows from the second by Proposition \ref{Form=RUC}. Fixed $f(x,y)$, $a$ and $B$, the function $f_a\in\RUC_u(B)$ is tame if and only if its extension to $\overline{B}$ is tame, where $\overline{B}$ is the closure in $S_y(M)$ of the image of $B$ under the compactification $M^\beta\to S_y(M)$. Then the second claim follows from Fact \ref{tame} and Corollary \ref{NIPcorollary}, taking $A'=Ga$. For this, one can see that Corollary \ref{NIPcorollary} holds true with $\overline{B}$ instead of $B^*$ or, alternatively, that $\overline{B}=B^*$ using that $M$ is $\aleph_0$-categorical. We show the latter. Clearly, $\overline{B}\subset B^*$. Let $p\in B^*$ and take $b$ a realization of $p$ in a separable elementary extension $M'$ of $M$. Let $\phi(z,y)$ be a formula, $c\in M^{|z|}$ and $\epsilon>0$. By $\aleph_0$-categoricity there is an isomorphism $\sigma\colon M'\to M$. Then $\tp(c)=\tp(\sigma c)$, so, by homogeneity, there is also an automorphism $g\in\Aut(M)$ with $d(c,g\sigma c)<\Delta_\phi(\epsilon)$. Hence $g\sigma b\in B$ and $|\phi(c,b)-\phi(c,g\sigma b)|<\epsilon$. We deduce that $p\in\overline{B}$.\end{proof}

\medskip

During the writing of this paper we came to know that, independently from us, A. Chernikov and P. Simon also noticed the connection between tameness in topology and $\NIP$ in logic, in the somehow parallel context of definable dynamics \cite{chesim15}. More on this connection has been elaborated by P. Simon in \cite{simRosenthal}.

In fact, it is surprising that the link was not made before, since the parallelism of these ideas in logic and topology is quite remarkable. As we have already said, $\NIP$ formulas were introduced by Shelah \cite{she71} in 1971, in the classical first-order context. He defined them by the lack of an \emph{independence property (IP)}, whence the name \emph{$\NIP$}. This independence property is the condition negated in the first item of Proposition~\ref{NIP}. In the classical first-order setting it can be read like this: a formula $\varphi(x,y)$ has IP if for some sequence of elements $(a_i)_{i<\omega}$ and every pair of non-empty finite disjoint subsets $I,J\subset\omega$, there is $b$ in some model that satisfies the formula $$\bigwedge_{i\in I}\varphi(a_i,y)\land\bigwedge_{j\in J}\neg\varphi(a_j,y).$$ In other words, $\varphi(x,y)$ has IP if for some $(a_i)_{i<\omega}$ the sequence $(\{b:\varphi(a_i,b)\})_{i<\omega}$ of the sets defined by $\varphi(a_i,y)$ on some big enough model of the theory is an \emph{independent sequence} in the sense of mere sets: all boolean intersections are non-empty.

In the introductory section 1.5 of the survey \cite{glameg13} on Banach representations of dynamical systems, Glasner and Megrelishvili write: \emph{\guillemotleft In addition to those characterizations already mentioned, tameness can also be characterized by the lack of an ``independence property'', where combinatorial Ramsey type arguments take a leading role [\dots]\guillemotright}. The characterization they allude to is Proposition~6.6 from Kerr and Li \cite{kerrli07}, and the independence property involved there can indeed be seen as a topological generalization of Shelah's IP (see also Fact \ref{kerrli} below). But the notion of independence is already present in the seminal work of Rosenthal from 1974 \cite{ros74}, where a crucial first step towards his dichotomy theorem implies showing that a sequence of subsets of a set $S$ with no convergent subsequence (in the product topology of $2^S$) admits a boolean independent subsequence. Moreover, as pointed out in \cite{simRosenthal}, the (not) independence property of Shelah, in its continuous form, appears unequivocally in the work of Bourgain, Fremlin and Talagrand \cite{BFT78}; see 2F.(vi).

On the other hand, this is not the first time that the concept of $\NIP$ is linked with a notion of another area. In 1992 Laskowski \cite{las92} noted that a formula $\varphi(x,y)$ has the independence property if and only if the family of definable sets of the form $\varphi(a,y)$ is a \emph{Vapnik--Chervonenkis class}, a concept coming from probability theory, and also from the 70's \cite{vapche71}. He then profited of the examples provided by model theory to exhibit new Vapnik--Chervonenkis classes.  In Section \ref{Examples} we shall do the same thing with respect to tame dynamical systems, complementing the analysis of the examples done by Ben Yaacov and Tsankov \cite[\textsection 6]{bentsa}.

\medskip

We end this section by pointing out that $\Tame_u(G)$ coincides, for Roelcke precompact Polish groups, with the restriction to $\UC(G)$ of the algebra $\Null(G)$ of \emph{null functions} on $G$. Null functions arise from the study of topological sequence entropy of dynamical systems, initiated in \cite{goo74}. A compact $G$-space $Y$ is null if its topological sequence entropy along any sequence is zero; we refer to \cite[\textsection 5]{kerrli07} for the pertinent definitions. We shall say that a function $f$ on an arbitrary $G$-space $X$ is null if it comes from a null compactification of $X$, and by Corollary 5.5 in \cite{kerrli07} this is equivalent to check that the cyclic $G$-space of $f$ is null. For compact $X$ this definition coincides with Definition 5.7 of the same reference (the $G$-spaces considered there are always compact), as follows from the statements 5.8 and 5.4.(2-4) thereof. The resulting algebra $\Null(X)$ is always a uniformly closed $G$-invariant subalgebra of $\Tame(X)$ (closedness is proven as for $\Tame(X)$; for the inclusion $\Null(X)\subset\Tame(X)$ compare \textsection{5} and \textsection{6} in \cite{kerrli07}).

The following fact is a rephrasing of the characterizations of Kerr and Li.

\begin{fact}\label{kerrli} A function $f\in\RUC(X)$ is null if and only if there are no real numbers $r<s$ such that for every $n$ one can find $(g_i)_{i<n}\subset G$ and $(x_I)_{I\subset n}\subset X$ such that $$\text{$f(g_ix_I)<r$ if $i\in I$ and $f(g_ix_I)>s$ if $i\notin I$}.$$\end{fact}
\begin{proof}
If $f$ is non-null then its extension to any compactification is non-null, and the existence of elements $r$, $s$ and, for every $n$, $(g_i)_{i<n}$ and $(x_I)_{I\subset n}$ as in the statement follows readily from Proposition 5.8 (and Definitions 5.1 and~2.1) in \cite{kerrli07}; we obtain the elements $x_I$ in the compactification, but we can approximate them by elements $\tilde{x}_I$ in $X$, since we only need that $f(g_i\tilde{x}_I)$ be close to $f(g_ix_I)$ for the finitely many indices $i<n$.

Conversely, if we have $r<s$ with the property negated in the statement, take $u,v$ with $r<u<v<s$ and consider the sets $A_0=\{p\in X_f:\tilde{f}(p)\leq u\}$, $A_1=\{p\in X_f:\tilde{f}(p)\geq v\}$; here, $\tilde{f}$ is the extension of $f$ to the cyclic $G$-space $X_f$. Then $A_0$ and $A_1$ are closed sets with arbitrarily large finite \emph{independence sets}. Hence by Proposition 5.4.(1) in the same paper there is an \emph{IN-pair} $(x,y)\in A_0\times A_1$, and by 5.8 we deduce that $f$ is non-null.
\end{proof}

When $X=M$ is an $\aleph_0$-categorical structure, it is immediate by $\emptyset$-saturation that a formula $f(x,y)$ is $\NIP$ on $[a]\times M$ if and only if $f_a$ is null. Thus $\Null_u(M)=\Tame_u(M)$, and by considering $M=\widehat{G}_L$ (and recalling Remark \ref{compactifications of G}) one gets $\Null_u(G)=\Tame_u(G)$ for every Roelcke precompact Polish $G$.

\noindent\hrulefill

\section{The hierarchy in some examples}\label{Examples}

Several interesting Polish groups are naturally presented as automorphism groups of well-known first-order structures. Moreover, most of these structures admit \emph{quantifier elimination}, which enables to describe their definable predicates in a simple way. As a result, the subalgebras of $\UC(G)$ that correspond to nice families of formulas can be understood pretty well in these examples.

Let $G$ be the automorphism group of an $\aleph_0$-categorical structure $M$. We recall from Proposition~\ref{Form=UC} that the functions $h$ in $UC(G)$ are exactly those of the form $h(g)=f(a,gb)$ for a formula $f(x,y)$ and tuples $a,b$ from $M$. Then $h$ factors through the orbit map $g\in G\mapsto gb\in [b]$. Bearing in mind Remark \ref{Y to Z}, it follows that
\begin{enumerate}
\item $h\in\AP(G)$ if and only if $f(x,y)$ is algebraic on $[a]\times[b]$ (Proposition \ref{AP formulas});
\item $h\in\WAP(G)$ if and only if $f(x,y)$ is stable on $[a]\times[b]$ (Fact \ref{stable is WAP});
\item $h\in\Tame_u(G)$ if and only if $f(x,y)$ is $\NIP$ on $[a]\times[b]$ (Proposition \ref{tameNIP}).
\end{enumerate}

However, a technical difficulty is that $f(x,y)$ may be a formula in infinite variables, whereas it is usually easier to work with predicates involving only finite tuples. This is especially the case in the study of classical structures, for which, moreover, the results in the literature are stated, naturally, for $\{0,1\}$-valued formulas in finitely many variables. In the following subsection we elaborate a way to deal with this difficulty. The reader willing to go directly to the examples may skip the details and retain merely the conclusion of Theorem \ref{cTame=Tame}.

\subsection{Approximation by formulas in finite variables} In this subsection $x$ and $y$ will denote variables of length $\omega$, and $M$ will be a $\emptyset$-saturated structure. Any formula $f(x,y)$ is, by construction, a uniform limit of formulas defined on finite sub-variables of $x,y$. Moreover, if $f(x,y)$ is, for instance, stable on $M^\omega\times M^\omega$, then one can uniformly approximate $f$ by stable formulas depending only on finite sub-variables of $x,y$. It suffices to take $n<\omega$ large enough so that, by uniform continuity, $|f(a,b)-f(a',b')|<\epsilon$ whenever $a_{<n}=a'_{<n}$ and $b_{<n}=b'_{<n}$; then define for example $f_n(x,y)=f(x',y')$, where $x'_{nk+i}=x_i$ and $y'_{nk+i}=y_i$ for all $i<n$, $k<\omega$.

However, if $f(x,y)$ is only known to be stable on $A\times B$ for some subsets $A,B\subset M^\omega$, then the previous simple construction does not ensure the stability of $f_n$. Besides, it may not be possible to find a formula stable on $M^\omega\times M^\omega$ that agrees with $f$ on $A\times B$.

In \cite{bentsa}, Proposition 4.7, a topological argument is given that permits to approximate $\WAP$ functions by stable formulas in finitely many variables. In what follows we give an alternative model-theoretic argument for this fact that can also be applied, in several cases, to $\NIP$ formulas.

In what follows, given a set $A\subset M$, the term $\acl(A)$ will denote the \emph{algebraic closure} of $A$, including \emph{imaginary elements} of $M$. The reader may wish to consult \cite[\textsection 10--11]{bbhu08} for an account of algebraic closure and imaginary sorts in continuous logic. Alternatively, and with no loss for the examples considered later, the reader may assume that $\acl(A)=A$.

\begin{defin}\label{detfs} Let $M$ be a metric structure, $f(x,y)$ a formula.
\begin{enumerate}
\item We will say that $f(x,y)$ is \emph{in finite variables} if there is $n<\omega$ such that $f(a,b)=f(a',b')$ whenever $a_{<n}=a'_{<n}$ and $b_{<n}=b'_{<n}$.
\item Let $a\subset M^\omega$ be a tuple and $B\subset M^\omega$ a definable set. We will say that $f(x,y)$ has \emph{definable extensions of types over finite sets on $a,B$} if for every large enough $n<\omega$ there are an $\acl(a_{<n})$-definable predicate $df(y)$ and a realization $a'$ of $\tp(a/a_{<n})$ (in some elementary extension of $M$) such that $$f(a',b)={df}(b)$$ for every $b\in B$.
\item We will say that $M$ has \emph{definable extensions of types over finite sets} if the previous condition is true on $a,M^\omega$ for every formula $f(x,y)$ and any $a\in M^\omega$.
\end{enumerate}
\end{defin}

\begin{lem}\label{stable implies definable extensions} Suppose $f(x,y)$ is stable on $A\times B$ for definable sets $A,B$. If $a\in A$, then $f(x,y)$ has definable extensions of types over finite sets on $a,B$.\end{lem}
\begin{proof} Let $n<\omega$. Since $A$ and $B$ are definable sets we can consider them as sorts in their own right (say, of an expanded structure $M'$), and consider $f(x,y)$ as a formula defined only on $A\times B$ (so $x$ and $y$ become $1$-variables of the corresponding sorts). Then $f$ is a stable formula in the usual sense of \cite{benusv10}, Definition 7.1, and we may apply the results thereof. More precisely, we can consider the $f$-type of $a$ over $C=\acl(a_{<n})$, call it $p\in S_f(C)$. Here, $p$ is an $f$-type (in the variable~$x$) in the sense of \cite{benusv10}, Definition 6.6. By Proposition 7.15 of the same paper, $p$ admits a definable extension $q\in S_f(M')$. Moreover, the type $q$ is consistent with $\tp(a/C)$, by the argument explained in \cite[\textsection 8.1]{benusv10}; note that, by adding dummy variables, each predicate $h(x)\in\tp(a/C)$ can be seen as a formula $h(x,y)$ (in the structure expanded with constants for the elements of $C$), which is trivially stable. Then it is enough to take for $a'$ any realization of~$q\cup\tp(a/C)$.\end{proof}

\begin{lem}\label{symmetry} Let $A,B\subset M^\omega$ be definable sets. Given a formula $f(x,y)$, define $\tilde{f}(y,x)=f(x,y)$. Then $f(x,y)$ is algebraic, stable or $\NIP$ on $A\times B$ if and only if so is $\tilde{f}(y,x)$ on $B\times A$.\end{lem}
\begin{proof} This is clear for the stable case. For the $\NIP$ case, the proof is as in \cite{sim14}, Lemma 2.5. If $f(x,y)$ is algebraic on $A\times B$ this means that $K=\{f_a|_B:a\in A\}$ is precompact in $\CC(B)$, so given $\epsilon>0$ there are $a_i\in A$, $i<n$, such that the functions $f_{a_i}|_B$ form an $\epsilon$-net for $K$. Let $I_j\subset\RR$, $j<m$, be a partition of the image of $f$ on sets of diameter less than $\epsilon$. For each function $\tau\colon n\to m$ let $b_\tau\in B$ be such that $f(a_i,b_\tau)\in I_{\tau(i)}$ for every $i<n$, if such an element exists. Then the functions $\tilde{f}_{b_\tau}|_A$ form a $3\epsilon$-net for $\tilde{K}=\{\tilde{f}_b|_A:b\in B\}\subset\CC(A)$. This shows that $\tilde{K}$ is also precompact, hence that $\tilde{f}$ is algebraic on $B\times A$.\end{proof}

In the following theorem we ask $M$ to be $\aleph_0$-categorical to ensure that the closed orbits we consider are definable sets. The addition of imaginary sorts does not affect the $\aleph_0$-categoricity of~$M$.

\begin{prop} Let $M$ be $\aleph_0$-categorical, $f(x,y)$ a formula, $a,b\in M^\omega$. Suppose either
\begin{enumerate}
\item $f(x,y)$ is algebraic on $[a]\times [b]$,
\item $f(x,y)$ is stable on $[a]\times [b]$, or
\item $M$ has definable extensions of types over finite sets, and $f(x,y)$ is $\NIP$ on $[a]\times [b]$.
\end{enumerate}
Then for every $\epsilon>0$ there is a formula $f_0(x,y)$ in finite variables such that $$\sup_{x\in [a],y\in [b]}|f(x,y)-f_0(x,y)|\leq\epsilon$$ and $f_0(x,y)$ is algebraic, stable or $\NIP$, respectively, on $[a]\times [b]$.
\end{prop}
\begin{proof}
Let $n$ be large enough, so that in particular $|f(u,v)-f(u',v)|\leq\epsilon/2$ for any tuples $u,u',v$ with $u_{<n}=u'_{<n}$. Using Lemma \ref{stable implies definable extensions} for cases (1) and (2) (remark that a formula algebraic on $A\times B$ is stable on $A\times B$) we have that in any case there are a formula $df(z,y)$, a parameter $c\in\acl(a_{<n})$ and a realization $a'$ of $\tp(a/a_{<n})$ in some elementary extension of $M$, such that $f(a',b')={df}(c,b')$ for every $b'\in [b]$.

Let $C$ be the set of realizations of $\tp(c/a_{<n})$. Since $c\in\acl(a_{<n})$, this set is compact, $a_{<n}$-definable and contained in the appropriate imaginary sort of $M$ (see \cite{bbhu08}, Exercise 10.8 and Proposition 10.6). Here, $a_{<n}$-definable means that $C$ is a definable set in the structure $M$ augmented with constants for the elements of $a_{<n}$ (that is, $d(x,C)$ is an $a_{<n}$-definable predicate), and hence we can quantify over $C$ in this augmented structure: in particular, $\sup_{z\in C}df(z,y)$ is an $a_{<n}$-definable predicate. This says that there is a formula $f'\colon M^n\times M^\omega\to\RR$ such that, for every $b'\in [b]$, $$f'(a_{<n},b')=\sup_{z\in C}df(z,b').$$ For any $a''$ with $a''_{<n}=a_{<n}$ we have $\sup_{y\in [b]}|f(a'',y)-df(c,y)|\leq\epsilon/2$, and the same is true if we replace $c$ by any $c'\in C$. We obtain $\sup_{y\in [b]}|f(a,y)-f'(a_{<n},y)|\leq\epsilon/2,$ and thus $$\sup_{x\in [a],y\in [b]}|f(x,y)-f'(x_{<n},y)|\leq\epsilon/2.$$

Now we consider each of the cases of the statement separately.
\begin{enumerate}
\item Let $(b_j)_{j<\omega}$ be an indiscernible sequence in $[b]$. By the hypothesis and Lemma \ref{symmetry}, the value of $f(a,b_j)$ is constant in $j$, and the same holds for $a'$ instead of $a$. Thus $df(c,b_j)$ is constant in $j$, and we can deduce that $df(z,y)$ is algebraic on $[c]\times [b]$. Since $C\subset [c]$, it follows that $f'(a_{<n},b_j)$ is constant too. We can conclude that $f'(x_{<n},y)$ is algebraic on $[a_{<n}]\times [b]$.

\item Since $f(x,y)$ is stable on (the definable sets) $[a]\times[b]$ and $M$ is $\emptyset$-saturated, no sequences $a'_i,b'_j$, in any elementary extension, with $\tp(a'_i)=\tp(a)$, $\tp(b'_j)=\tp(b)$, can witness the order property for $f(x,y)$. Hence, the function $f_{a'}\in\CC([b])$ is $\WAP$. Since $f_a'=df_c$ on $[b]$, it follows that $df(z,y)$ is stable on $[c]\times[b]$. Since $C$ is compact, it is not difficult to deduce that $f'(x_{<n},y)$ is stable on $[a_{<n}]\times [b]$. For example, we know that $\max_{l<k}df_{c_l}$ is in $\WAP([b])$ for every $(c_l)_{l<k}\subset C$, and $f'_{a_{<n}}|_{[b]}$ is a uniform limit of functions of this form.

\item Here, if $(b_j)_{j<\omega}\subset [b]$ is an indiscernible sequence and $g$ is an automorphism of $M$, the sequence $(df(gc,b_j))_{i<\omega}$ must converge in $\RR$. Indeed, $df(gc,b_j)=f(a',g^{-1}b_j)$, so the claim follows from the fact that $f(x,y)$ is $\NIP$ on $[a]\times[b]$ and $(g^{-1}b_j)_{j<\omega}$ is also indiscernible. By uniform continuity and a density argument, the same is true if we replace $gc$ with any $c'\in[c]$. We deduce that $df(z,y)$ is $\NIP$ on $[c]\times [b]$. As in the previous item, this implies that $f'(x_{<n},y)$ is $\NIP$ on $[a_{<n}]\times [b]$.
\end{enumerate}

This is half what we intended. To complete the proof it suffices to apply the same construction to the formula $\tilde{f}'(y,x)=f'(x_{<n},y)$. We obtain a formula $f''(y_{<m},x)$; we define $f_0(x,y)=f''(y_{<m},x)$, then $f_0(x,y)$ is in finite variables and satisfies the other conditions of the statement.
\end{proof}

\begin{question} Is the previous result true in the $\NIP$ case without the assumption on $M$?\end{question}

We remark that a $\{0,1\}$-valued formula is necessarily in finite variables. Also, any formula with finite range can be written as a linear combination of $\{0,1\}$-valued formulas. If $M$ is \emph{classical} $\aleph_0$-categorical, then, conversely, any formula in finite variables has finite range, since it factors through the finite space $M^k\sslash G$ for some $k<\omega$. For $G=\Aut(M)$ it follows that $\UC(G)$ is the closed algebra generated by the functions of the form $g\mapsto f(a,gb)$ where $a,b$ are parameters and $f(x,y)$ is a \emph{classical} (i.e.\ $\{0,1\}$-valued) formula.

Let us define $c\Tame_u(G)$ (respectively, $c\AP(G)$, $c\WAP(G)$) as the closed subalgebra of $\UC(G)$ generated by the functions of the form $g\mapsto f(a,gb)$ for $\{0,1\}$-valued $\NIP$ (resp., algebraic, stable) formulas $f(x,y)$. That is, these are the algebras generated by classical formulas of the appropriate corresponding kind. Here, assuming $M$ is classical $\aleph_0$-categorical, it is indifferent to ask $f(x,y)$ to be $\NIP$ only on $[a]\times [b]$ or in its whole domain, since one can easily modify $f$ so that it be $\NIP$ (resp., algebraic, stable) everywhere, without changing the function $g\mapsto f(a,gb)$. Indeed, one can assume that $a,b\in M^k$ for some $k<\omega$, then set $f$ to be $0$ outside $[a]\times[b]$ (since $M^k$ is discrete and $[a]\times[b]$ is definable, the modified $f$ is still definable).

From the previous proposition and discussion we obtain the following conclusion, which extends Theorem~5.4 in \cite{bentsa}.

\begin{theorem}\label{cTame=Tame}
Let $M$ be a classical $\aleph_0$-categorical structure, $G$ its automorphism group. Then $c\AP(G)=\AP(G)$ and $c\WAP(G)=\WAP(G)$. If $M$ has definable extensions of types over finite sets, then also $c\Tame_u(G)=\Tame_u(G)$.
\end{theorem}

As we will see shortly, the assumption that $M$ has definable extension of types over finite sets is satisfied in many interesting cases. The following is a useful sufficient condition.

\begin{lem}\label{defts via invariant types} Suppose $M$ is classical, $\aleph_0$-categorical, and that for every $a\in M^\omega$ and $n<\omega$ there is a type $p\in S_x(M)$ such that $p$ extends $\tp(a/a_{<n})$ and $p$ is $a_{<n}$-invariant (i.e.\ $p$ is fixed under all automorphisms of $M$ fixing the tuple $a_{<n}$). Then $M$ has definable extensions of types over finite sets.\end{lem}
\begin{proof}
Let $a\in M^\omega$, $n<\omega$; take $p$ as in the hypothesis of the lemma, $a'$ a realization of $p$. Given a formula $f(x,y)$, the function $df$ defined by $df(b)=f(a',b)$ is $a_{<n}$-invariant. Since $M$ is classical $\aleph_0$-categorical, the structure $M$ expanded with constants for the elements of $a_{<n}$ is $\aleph_0$-categorical too (see \cite{tenzie}, Corollary 4.3.7). It follows that $df(y)$ is an $a_{<n}$-definable predicate, hence the conditions of Definition \ref{detfs} are satisfied.\end{proof}

\medskip

\subsection{The examples} We describe the dynamical hierarchy of function algebras for the automorphism groups of some well-known (unstable) $\aleph_0$-categorical structures. We start with the oligomorphic groups $\Aut(\QQ,<)$, $\Aut(RG)$ and $\Homeo(2^\omega)$.

The unique countable dense linear order without endpoints, $(\QQ,<)$, admits quantifier elimination (see \cite[\textsection 3.3.2]{tenzie}). This implies, for $G=\Aut(\QQ,<)$, that $\UC(G)$ is the closed unital algebra generated by the functions of the form $g\mapsto (a=gb)$ and $g\mapsto (a<gb)$ for elements $a,b\in\QQ$ ---where we think of the classical predicates $x=y$ and $x<y$ as $\{0,1\}$-valued functions. The formula $x<y$ is $\NIP$ (and $x=y$ is of course stable), whence we deduce that every $\UC$ function is tame. On the other hand, $x<y$ is unstable, so $g\mapsto (a<gb)$ is not $\WAP$ (in fact, as follows from \cite{bentsa}, Example 6.2, $\WAP(G)$ is precisely the unital algebra generated by the functions of the form $g\mapsto (a=gb)$).

Now suppose $f(x,y)$ is a formula algebraic on $[a]\times [b]$. For slight convenience we may assume, by Theorem \ref{cTame=Tame}, that $f$ is classical and the tuples involved are finite. For tuples $c,d$, let us write $c<d$ to mean that every element of the tuple $c$ is less than every element of $d$. Let $b'\in[b]$. We can choose a sequence of tuples $(a^i)_{i<\omega}$ in $\QQ$ (or in an elementary extension if we did not assume the tuples are finite) such that $a\simeq a^i$ as linear orders, $a=a_0$, $b<a^1$, $b'<a^1$, and $a^i<a^j$ if $i<j$. By quantifier elimination, the type of a tuple depends only on its isomorphism type as a linear order; hence $(a^i)_{i<\omega}$ is an indiscernible sequence. By the hypothesis on $f$ we have that $(f(a^i,b))_{i<\omega}$ is constant, and the same with $b'$ instead of $b$. But, again by quantifier elimination, $f(a^1,b)=f(a^1,b')$. It follows that $f(a,b)=f(a,b')$. We have thus shown that $g\mapsto f(a,gb)$ is constant, and can deduce that $G$ is $\AP$-trivial. (In fact, as is well-known, $G$ is extremely amenable (\cite{pesExtreme}), which is a much stronger property: if $f\in\AP(G)$ then the compact $G$-space $\overline{Gf}$ must have a fixed point; since the action of $G$ on $\overline{Gf}$ is by isometries we conclude that $\overline{Gf}$ is a singleton, i.e. that $f$ is constant.)

Putting these conclusions together we get the following (where $\RR$ stands for the algebra of constant functions on $G$).

\begin{cor} For $G=\Aut(\QQ,<)$ we have $\RR=\AP(G)\subsetneq\WAP(G)\subsetneq\Tame_u(G)=\UC(G)$.\end{cor}

\medskip

The situation is different for the random graph $RG$, the unique countable, homogeneous, universal graph. It has quantifier elimination, which in this case implies that $\UC(\Aut(RG))$ is the closed unital algebra generated by the functions of the form $g\mapsto (a=gb)$ and $g\mapsto (a\mathrel{R}gb)$ (where $R$ denotes the adjacency relation of the graph). Also, stable formulas on $[a]\times [b]$ are again exactly those expressible in the reduct of $RG$ to the identity relation (\cite{bentsa}, Example 6.1). But in this case no other formula is $\NIP$ on $[a]\times [b]$.

\begin{lem} On the random graph, every classical $\NIP$ formula is stable.\end{lem}
\begin{proof} Theorem 4.7 in \cite[Ch. II]{sheClassification} shows that if there is an unstable $\NIP$ formula then there is a formula with the \emph{strict order property}. The theory of the random graph, being \emph{simple}, does not admit a formula with the strict order property; see \cite{tenzie}, Corollary 7.3.14 and Exercise 8.2.4.\end{proof}

It follows for $G=\Aut(RG)$ that $c\Tame_u(G)=c\WAP(G)$. Now we argue that $RG$ has definable extensions of types over finite sets, whence $\Tame_u(G)=\WAP(G)$ by Theorem \ref{cTame=Tame}. For any $a\in {RG}^\omega$ and $n<\omega$, the \emph{free amalgam} of $a$ and $RG$ over $a_{<n}$ is a graph containing $RG$ and a copy $a'\simeq a$ such that $a_{<n}=a'_{<n}$ and, for every $i\geq n$, $a_i$ is not $R$-related to any element of $RG$ outside $a_{<n}$. The homogeneity and universality of $RG$ ensure that such a copy $a'$~is realized as a tuple in some elementary extension of $RG$. Since $RG$ has quantifier elimination it is clear that $a'$ realizes $\tp(a/a_{<n})$ and that the type $\tp(a'/RG)$ is $a_{<n}$-invariant, thus Lemma \ref{defts via invariant types} applies. We can conclude that $\Tame_u(G)$ is the closed unital algebra generated by the functions of the form $g\mapsto (a=gb)$, $a,b\in RG$. An example of a non-tame function in $\UC(G)$ is of course $g\mapsto (a\mathrel{R}gb)$.

If $f(x,y)$ is a formula algebraic on $[a]\times [b]$ and $b'\in [b]$, we can take a sequence $(a^i)_{i<\omega}$ of disjoint copies of $a$ such that $a=a^0$ and no element of $a^i$, $i\geq 1$, is $R$-related to an element of $b$, $b'$ nor $a^j$ for $j\neq i$. It follows by quantifier elimination that $(a^i)_{i<\omega}$ is indiscernible and that $f(a^1,b)=f(a^1,b')$. Since, by hypothesis, $(f(a^i,b))_{i<\omega}$ and $(f(a^i,b'))_{i<\omega}$ must be constant, we obtain $f(a,b)=f(a,b')$. Thus $g\mapsto f(a,gb)$ is constant.

\begin{cor} For $G=\Aut(RG)$ we have $\RR=\AP(G)\subsetneq\WAP(G)=\Tame_u(G)\subsetneq\UC(G)$.\end{cor}

\medskip

The group $G=\Homeo(2^\omega)$ of homeomorphisms of the Cantor space, carrying the compact-open topology, can be identified naturally with the automorphism group of the boolean algebra $\BB$ of clopen subsets of $2^\omega$, with the topology of pointwise convergence. Up to isomorphism, $\BB$ is the unique countable atomless boolean algebra. We consider it as a structure in the language of boolean algebras, that is, we have basic functions $\wedge,\vee\colon\BB^2\to\BB$ and $\neg\colon\BB\to\BB$ for meet, joint and complementation in the algebra, and constants $0$ and $1$ for the minimum and maximum of~$\BB$. In~this language $\BB$ admits quantifier elimination (see \cite{poizatCours}, Th\'eor\`eme 6.21). This means that two tuples $c,d$ of the same length have the same type over $\emptyset$ if and only if $c\simeq d$ (i.e.\ the map $c_i\mapsto d_i$ extends to an isomorphism of the generated boolean algebras). It also implies that $\UC(G)$ is the algebra generated by the functions of the form $g\mapsto (0=t(a,gb))$, where $t(x,y)$ is a boolean term in finite variables, i.e.\ a function $\BB^n\times\BB^m\to\BB$ constructed with the basic boolean operations $\wedge,\vee,\neg$.

Here it is easy to see that $0=x\wedge y$ is not $\NIP$ on $[a]\times[b]$ (for $a,b\notin\{0,1\}$), so the function $g\mapsto (0=a\wedge gb)$ is not tame. With this in mind, and following the idea of \cite{bentsa}, Example 6.3, one sees the following.

\begin{lem} On the countable atomless boolean algebra, every classical $\NIP$ formula is stable.\end{lem}
\begin{proof}
Let $f(x,y)$ be a classical formula. If $f$ is not stable, then it is not stable on $[a]\times [b]$ for some tuples $a$, $b$ from $\BB$. We may modify $f(x,y)$, $a$ and $b$ (without changing the function $g\mapsto f(a,gb)$) so that the elements of the tuple $a$ form a finite partition of $1$, and the same for $b$.

Say $a\in\BB^n$, $b\in\BB^m$. In \cite{bentsa}, Example 6.3, it is shown that $f(x,y)$ is unstable on $[a]\times [b]$ if and only if there is $b'\in[b]$ such that $f(a,b)\neq f(a,b')$ but the $xy$-tuples $ab$ and $ab'$ satisfy the same formulas of the form $t(x)=s(y)$ for boolean terms $t,s$. Moreover, it is shown that in this case one can choose $b'$ so that, for some indices $i_0,i_1<n$, $j_0,j_1<m$, we have (possibly changing $b$ by a conjugate):
\begin{enumerate}
\item $a_{i_0}\wedge b_{j_0}=0$, $a_{i_0}\wedge b_{j_1}\neq 0$, $a_{i_1}\wedge b_{j_0}\neq 0$ and $a_{i_1}\wedge b_{j_1}\neq 0$;
\item $a_i\wedge b'_j\neq 0$ for $i\in\{i_0,i_1\}$ and $j\in\{j_0,j_1\}$;
\item $a_i\wedge b_j=0$ if and only if $a_i\wedge b'_j=0$, for every pair $(i,j)\neq (i_0,j_0)$;
\item $f(a,b)\neq f(a,b')$.
\end{enumerate}

Now we fix an arbitrary $l<\omega$ and choose a partition $c_0,\dots,c_l$ of $a_{i_1}\wedge b_{j_1}$. For each $k<l$ we let $a^k_{i_0}=a_{i_0}\vee c_k$ and $a^k_{i_1}=a_{i_1}\wedge\neg c_k$. We also let $a^k_i=a_i$ for every $i\notin\{i_0,i_1\}$, thus defining a tuple $a^k\in [a]$. Similarly, for every $K\subset l$ we let $b^K_{j_0}=b_{j_0}\vee(\bigvee_{k\in K}c_k)$ and $b^K_{j_1}=b_{j_1}\wedge\neg(\bigvee_{k\in K}c_k)$. We let $b^K_j=b_j$ for $j\notin\{j_0,j_1\}$, and this defines a tuple $b^K\in [b]$. By quantifier elimination, the type of the tuple $a^kb^K$ is determined by the set of pairs $i,j$ such that $a^k_i\wedge b^K_j=0$. It follows that
\begin{equation*}
\text{$f(a^k,b^K)=f(a,b')$ if $k\in K$, and $f(a^k,b^K)=f(a,b)$ if $k\notin K$.}
\end{equation*}
Hence $f(x,y)$ is not $\NIP$ on $[a]\times [b]$.
\end{proof}

The lemma shows that $c\Tame_u(G)=c\WAP(G)$ for $G=\Aut(\BB)$, and that this algebra is generated by the functions of the form $g\mapsto (a=gb)$. Indeed, the proof shows that a $\NIP$ formula on $[a]\times [b]$ is a combination of formulas of the kind $t(x)=s(y)$ for boolean terms $t,s$; then simply note that $s(gb)=gs(b)$, so the function on $G$ associated to a formula of latter kind is $g\mapsto(c=gd)$ where $c=t(a)$ and $d=s(b)$. Next we show that $\BB$ has definable extensions of types over finite sets, in order to conclude, by Theorem \ref{cTame=Tame}, that these functions actually generate $\Tame_u(G)$.

Let $a\in\BB^\omega$, $n<\omega$; with no loss of generality we may assume that $a_{<n}$ is a partition of $1$. We consider the free amalgam of $a$ and $\BB$ over $a_{<n}$, which is a boolean algebra generated by $\BB$ together with a copy $a'\simeq a$ such that $a'_{<n}=a_{<n}$ and, for every $i<n$, every $d\in\BB$ and every $c$ in the boolean algebra generated by $a'$, we have $c\wedge a_i\wedge d\neq 0$ unless $c\wedge a_i=0$ or $a_i\wedge d=0$. Such a copy $a'$ is realized as a tuple in some elementary extension of $\BB$, and the type $\tp(a'/\BB)$ is clearly $a_{<n}$-invariant. By Lemma \ref{defts via invariant types}, $\BB$ has definable extensions of types over finite sets.

Finally, as with $\Aut(\QQ,<)$ and $\Aut(RG)$, we show that every $\AP$ function on $\Homeo(2^\omega)$ is constant. If $f(x,y)$ is algebraic on $[a]\times [b]$ and $b'\in [b]$, we can find copies $a^i$ of $a$ such that $a=a^0$ and each $a^i$, $i\geq 1$, forms a free amalgam with $B_i$ over $\emptyset$, where $B_i$ is the algebra generated by $b$, $b'$ and all $a^j$, $j<i$. That is, $c\wedge d\neq 0$ for every non-zero $d\in B_i$ and every non-zero $c$ in the algebra generated by $a^i$. Then $(a^i)_{i<\omega}$ is indiscernible and $f(a^1,b)=f(a^1,b')$. By hypothesis $(f(a^i,b))_{i<\omega}$ and $(f(a^i,b'))_{i<\omega}$ are constant, and therefore $f(a,b)=f(a,b')$.

\begin{cor} For $G=\Homeo(2^\omega)$ we have $\RR=\AP(G)\subsetneq\WAP(G)=\Tame_u(G)\subsetneq\UC(G)$.\end{cor}

\medskip

The previous examples come from classical structures; we consider now a purely metric one: the Urysohn sphere $\mathbb{U}_1$. This is, up to isometry, the unique separable, complete and homogeneous metric space of diameter 1 that is universal for countable metric spaces of diameter at most 1: any such metric space can be embedded in $\mathbb{U}_1$. As a metric structure with no basic predicates (other than the distance), $\mathbb{U}_1$ is $\aleph_0$-categorical and has quantifier elimination, which in this case means that the type of a tuple $b$ depends only on the isomorphism class of $b$ as a metric space; see \cite[\textsection 5]{usvGeneric}. It also says that $\UC(\Iso(\mathbb{U}_1))$ is generated by the functions of the form $g\mapsto d(a,gb)$ for $a,b\in\mathbb{U}_1$.

We show that $\Iso(\mathbb{U}_1)$ is $\Tame_u$-trivial.

\begin{theorem}\label{TameIsoU is trivial} Every function in $\Tame_u(\Iso(\mathbb{U}_1))$ is constant.\end{theorem}
\begin{proof} Suppose $f(x,y)$ is not constant on $[a]\times [b]$, so we have $f(a,b)\neq f(a',b')$ where $a\simeq a'$ and $b\simeq b'$ as metric spaces. We will need to assume that the elements of $a$ are separated enough from the elements of $b$, (the same for $a',b'$), and that the metric space $ab$ is similar to $a'b'$; so we precise and justify this. Let $0<\epsilon<1$. Note first that, by the universality and homogeneity of the Urysohn sphere, for any tuples $x,y$ in $\mathbb{U}_1$ (or in an elementary extension thereof) we can find $\tilde{y}$ in an elementary extension such that $y\simeq \tilde{y}$, $d(y_n,\tilde{y}_n)=\epsilon$ and $d(x_n,\tilde{y}_m)=(d(x_n,y_m)+\epsilon)\wedge 1$ for all coordinates $n,m$ (where $r\wedge s$ denotes $\min(r,s)$). Note secondly that finitely many iterations of the process of replacing $y$ by $\tilde{y}$ eventually end with $d(x_n,y_m)=1$ for all $n,m$.

If we chose $\epsilon$ small enough and do one iteration of the previous process for $xy=ab$, by continuity of $f$ we can assume (replacing $ab$ by $a\tilde{b}$) that $$d(a_n,a_m)\leq d(a_n,b_k)+d(b_k,a_m)-\epsilon,\ d(b_n,b_m)\leq d(b_n,a_k)+d(a_k,b_m)-\epsilon$$ for all $n,m,k$ ---and still $f(a,b)\neq f(a',b')$. So we have separated the elements of $a$ from those of $b$, and we do the same for $a',b'$.

Next we iterate the process described above for $\epsilon/2$, starting with $xy=ab$, thus producing a finite sequence of copies of $b$, the last copy $\tilde{b}$ verifying $d(a_n,\tilde{b}_m)=1$ for all $n,m$. We do the same starting with $a',b'$, finishing with a copy $\tilde{b}'$ with the analogous property. Since the theory of $\mathbb{U}_1$ has quantifier elimination, we have $f(a,\tilde{b})=f(a',\tilde{b}')$. So $f$ differs in two consecutive steps of the process, and by replacing our tuples $ab$, $a'b'$ by these consecutive tuples we may assume also that $a=a'$ and $|d(a_n,b_m)-d(a_n,b'_m)|\leq\epsilon/2$ for all $n,m$.

With the previous assumptions in mind, we now construct a metric space containing a sequence $(a^i)_{i<\omega}$ of different copies of $a$ and, for each $I\subset\omega$, a copy $b^I$ of $b\simeq b'$ such that $a^ib^I\simeq ab$ if $i\in I$ and $a^ib^I\simeq ab'$ if $i\notin I$. For $i\neq j$, $I\neq J$ and each $n,m$ we define $d(a^i_n,a^j_m)=(d(a_n,a_m)+\epsilon/2)\wedge 1$, $d(b^I_n,b^J_m)=(d(b_n,b_m)+\epsilon/2)\wedge 1$. The triangle inequalities are satisfied; for example, for $a^i_n,a^j_m,b^I_k$, ${i\neq j}$, $i\in I$, we have $d(a^i_n,a^j_m)=(d(a_n,a_m)+\epsilon/2)\wedge 1\leq d(a_n,b_k)+d(b_k,a_m)-\epsilon/2\leq d(a^i_n,b^I_k)+d(b^I_k,a^j_m)$, and also $d(a^i_n,b^I_k)\leq(d(a^i_n,a^i_m)+\epsilon/2)\wedge 1+d(a^i_m,b^I_k)-\epsilon/2\leq d(a^i_n,a^j_m)+d(a^j_m,b^I_k)$. The other inequalities are proved similarly.

By the universality of the Urysohn sphere we can assume that the tuples $a^i$ lie in $\mathbb{U}_1$, the tuples $b^I$ in some elementary extension. By quantifier elimination, $f(a^i,b^I)=f(a,b)$ if $i\in I$ and $f(a^i,b^I)=f(a,b')$ if $i\notin I$. This shows that $f(x,y)$ is not $\NIP$ on $[a]\times [b]$. It follows that every tame function of the form $g\mapsto f(a,gb)$ is constant, which proves the theorem.
\end{proof}

In Question 7.10 from \cite{glameg10} it was asked whether the algebra $\Tame(G)$ separates points and closed subsets of $G$ for every Polish group $G$. As we have seen, this can fail drastically for the algebra $\Tame_u(G)$. Unfortunately, we do not know how big may the gap between $\Tame(G)$ and $\Tame_u(G)$ be.

\begin{question} Are there tame non-constant functions on $\Iso(\mathbb{U}_1)$? Is there a way to \emph{regularize} a (non-constant) function $f\in\RUC(G)$ to get (a non-constant) $\tilde{f}\in\UC(G)$, in such a manner that tameness is preserved?\end{question}

\medskip

Finally, we consider the group $G=H_+[0,1]$ of increasing homeomorphisms of $[0,1]$ with the compact-open topology ---which coincides on $G$ with those of pointwise or uniform convergence. In spite of not being naturally presented as an automorphism group of some $\aleph_0$-categorical metric structure, this group is Roelcke precompact. See \cite{uspComp}, Example 4.4, for a description of its Roelcke compactification.

The following result was explained to us by M. Megrelishvili.

\begin{theorem}\label{megrelishvili}
$\UC(H_+[0,1])\subset\Tame(H_+[0,1])$.
\end{theorem}

See \cite{glameg14}, Theorem 8.1. As remarked there, the inclusion is strict: the function $f\colon G\to\RR$ given by $f(g)=g(1/2)$, for example, is tame (it comes from the Helly space, which is a Rosenthal compactification of $G$) but not left uniformly continuous: $\sup_g|f(gh)-f(g)|=1$ for any $h\in G$ with $h(1/2)\neq 1/2$. In fact, $f$ is even null: it is clear that, for reals $r<s$, there are no increasing functions $g_0,g_1\in [0,1]^{[0,1]}$ and elements $x_{\{0\}},x_{\{1\}}\in [0,1]$ such that $g_i(x_I)<r$ if $i\in I$ and $g_i(x_I)>s$ if $i\notin I$. One deduces that $\UC(G)\subsetneq\Null(G)$.

Additionally, as we have already recalled, the celebrated result of \cite{meg01} says that $H_+[0,1]$ is $\WAP$-trivial. On the other hand, one of the main results of \cite{glameg08} (Theorem 8.3) is the stronger fact that $H_+[0,1]$ is $\SUC$-trivial. In turn, this allows the authors to deduce that $\Iso(\mathbb{U}_1)$ is also $\SUC$-trivial (\cite[\textsection 10]{glameg08}). By our Theorem \ref{main} we can recover these facts directly from the $\WAP$-triviality of these groups, and extend the conclusion to another interesting Roelcke precompact Polish group that is also known to be $\WAP$-trivial.

\begin{cor} The groups $H_+[0,1]$ and $\Iso(\mathbb{U}_1)$ are $\SUC$-trivial. The same is true for the homeomorphism group of the Lelek fan.\end{cor}
\begin{proof} The $\WAP$-triviality of $\Iso(\mathbb{U}_1)$ was first observed in \cite{pesUrysohn}, Corollary 1.4, using the analogous result for $H_+[0,1]$; an alternative proof is given in \cite{bentsa}, Example 6.4, and of course also follows from Theorem \ref{TameIsoU is trivial} above. For the homeomorphism group of the Lelek fan, $\WAP$-triviality was proven in \cite{bentsa}: see the discussion after Corollary 4.10 and the references therein.\end{proof}

The previous facts about the group $H_+[0,1]$ lead to an interesting model-theoretic example, addressed in the following corollary.

If $f(x,y)$ is a formula in the variables $x,y$ (of arbitrary length), let us say that $f(x,y)$ is \emph{separated} if it is equivalent to a continuous combination of definable predicates $f_i(z_i)$ where, for each $i$, $z_i=x$ or $z_i=y$. Equivalently, $f(x,y)$ is separated if it factors through the product of type spaces ${S_x(\emptyset)\times S_y(\emptyset)}$ (by the Stone--Weierstrass theorem, the continuous functions on a product $X\times Y$ of compact Hausdorff spaces is the closed algebra generated by the continuous functions that depend only on $X$ or on $Y$). Of course, separated formulas are stable. Let us say that a structure is \emph{purely unstable} if every stable formula $f(x,y)$ is separated. No infinite classical structure can be purely unstable, since the identity relation $x=y$ is always stable and never separated.

\begin{cor} The $\aleph_0$-categorical structure $M=\widehat{G}_L$ associated to $G=H_+[0,1]$ is purely unstable and $\NIP$.\end{cor}
\begin{proof}
Of course, $\WAP$-triviality implies that $M$ is purely unstable: if $f(x,y)$ is stable and $a,b$ are parameters, then the function $g\mapsto f(a,gb)$ belongs to $\WAP(G)$ and so is constant. It follows that the value of $f$ on $a,b$ only depends on $[a],[b]$, that is, on $\tp(a),\tp(b)$ since $M$ is $\aleph_0$-categorical; hence $f(x,y)$ is separated.

On the other hand, Theorem \ref{megrelishvili} and Proposition \ref{tameNIP} imply that every formula $f(x,y)$ with $|y|=1$ is NIP. A well-known argument (see for example Proposition 2.11 in \cite{sim14}, which adapts easily to the metric setting), shows that then every formula is NIP.
\end{proof}

\medskip

We finish with a remark relating sections \ref{wap=suc} and \ref{tame=nip} of this paper. Since reflexive-representable functions correspond to stable formulas and Rosenthal-representable functions correspond to $\NIP$ formulas, it is not surprising that, as we have seen, the natural intermediate subalgebra of Asplund-representable functions collapses to one of the other two: on the model-theoretic side, there is no known natural notion between stable and $\NIP$. However, one might be slightly surprised to find that $\WAP=\Asp$ rather than $\Asp=\Tame_u$ (although, in fact, this was already known for $G=H_+[0,1]$). Indeed, Asplund and Rosenthal Banach spaces were once difficult to distinguish, with the first examples coming in the mid-seventies from independent works of James and of Lindenstrauss and Stegall. It is thus worthy to remark that, via our results and the Banach space construction of Glasner and Megrelishvili \cite{glameg12} (Theorem 6.3), every $\NIP$ unstable $\aleph_0$-categorical structure yields an example of a Rosenthal non-Asplund Banach space.

\noindent\hrulefill

\vspace{-.02in}

\bibliographystyle{amsalpha}
\bibliography{mybiblio}

\end{document}